\documentclass[a4paper,10pt, reqno]{amsart}

\makeatletter
\@namedef{subjclassname@2020}{%
  \textup{2020} Mathematics Subject Classification}
\makeatother

\usepackage[
  margin=30mm,
  marginparwidth=25mm,     
  marginparsep=2mm,       
  bottom=25mm,
  ]{geometry}

\usepackage[bbgreekl]{mathbbol}
\usepackage{amsfonts}
\usepackage{latexsym,amssymb,amsthm,mathrsfs,amsmath,amscd,enumerate,enumitem, amscd,color}

\usepackage{tikz}

\DeclareSymbolFontAlphabet{\mathbb}{AMSb}
\DeclareSymbolFontAlphabet{\mathbbl}{bbold}

\newtheorem{thm}{Theorem}[section]
 
 \newtheorem{lem}[thm]{Lemma}
 
\theoremstyle{definition}

 \theoremstyle{remark}

\usepackage{hyperref}

\newcommand{\supp}{\mathop{\mathrm{supp}}}

\numberwithin{equation}{section}
\allowdisplaybreaks

\begin{document}


\title[]{The logarithmic p-Laplacian on hyperbolic spaces}
\author[J. J. Betancor]{J. J. Betancor}

\author[L. Rodr\'{\i}guez-Mesa] {L. Rodr\'{\i}guez-Mesa}

\address{
	Departamento de An\'alisis Matem\'atico, Universidad de La Laguna,\newline
	Campus de Anchieta, Avda. Astrof\'isico S\'anchez, s/n,\newline
	38721 La Laguna (Sta. Cruz de Tenerife), Spain}
\email{jbetanco@ull.es, lrguez@ull.edu.es}

\thanks{The authors are partially supported by the Grant PID2023-148028NB-I00 funded by MICIU/AEI/10.13039/501100011033 and by ERDF/EU”.
}




\begin{abstract}
In this paper, the logarithmic $p$-Laplacian operator $\log (-\Delta_{\mathbb H ^n})_p$ on the hyperbolic space $\mathbb H^n$, with $n\geq 2$, is introduced. We prove that if $f$ is a locally Lipschitz function of exponent $\alpha \in (0,1)$ with compact support in $\mathbb H^n$, then, for a suitable constant  $A_{n,p}>0$,  
$$
\lim_{s\rightarrow 0^+}(-\Delta_{\mathbb H ^n})_p^sf(x)=A_{n,p}|f(x)|^{p-2}f(x),\quad x\in \mathbb H^n,
$$  
where $(-\Delta_{\mathbb H ^n})_p^s$ denotes the $s$-fractional $p$-Laplacian on $\mathbb H^n$. We establish a pointwise integral representation for the operator $\log (-\Delta_{\mathbb H ^n})_p=\frac{d}{ds}(-\Delta_{\mathbb H^n})_p^s\,_{|s=0}$. Furthermore, we show that $\log (-\Delta_{\mathbb H ^n})_p$ can be realized as the solution of a suitable extension problem and provide an extension theorem that yields the operator $\log (-\Delta)_p$ in $\mathbb R^n$. To the best of our knowledge, this property has not been established for the Euclidean logarithmic $p$-Laplacian $\log (-\Delta)_p$. 
\end{abstract}

\maketitle

\section{Introduction}
Recently the logarithmic of the Laplace operator $\log(-\Delta)$ in $\mathbb R^n$ has been studied in \cite{CW}. This operator is defined by
$$
\log (-\Delta )f=\frac{d}{ds}(-\Delta)^sf_{|s=0},
$$
where $(-\Delta)^s$ denotes the $s$-power of the Laplacian operator for every $s\in (0,1)$. The fractional Laplacian operator $(-\Delta)^s$ admits several equivalent definitions (\cite{KW}) (by using Fourier transformation, semigroup of operators and singular integrals, for instance). Caffarelli and Silvestre in their celebrated paper \cite{CS} obtained the fractional Laplacian operator $(-\Delta)^s$ as the solution of an extension problem. It is remarkable that the operators $(-\Delta)^s$, $s\in (0,1)$, and $\log(-\Delta)$ are nonlocal. As it is well-known nonlocal operators have emerged related to analysis, geometry, applied mathematics and probability (see \cite{JX, LPW, L, RS} and \cite{SV}).

A pointwise representation of $\log (-\Delta)$ was obtained in \cite[Theorem 1.1]{CW}. In \cite[Theorem 1.2]{CHW} it is proved that the operator $\log (-\Delta)$ can be obtained as the solution of an extension problem. 

For every $p>1$ the $p$-th Laplacian operator $(-\Delta)_p$ is defined by 
$$
(-\Delta)_pf=-{\rm div} (|\nabla f|^{p-2}\nabla f).
$$
It is clear that $(-\Delta)_2=-\Delta$. The fractional $p$-Laplace operator $(-\Delta)_p^s$, $s\in (0,1)$, is given by
$$
(-\Delta)_p^sf(x)=c_{n,s,p}{\rm P.V.}\int_{\mathbb R^n}\frac{|f(x)-f(y)|^{p-2}(f(x)-f(y))}{|x-y|^{n+sp}}dy,
$$
where P.V. means the Cauchy principal value of the integral. The constant $c_{n,s,p}$ is choosen such that we have the limits
$$
\lim_{s\rightarrow 0^+}(-\Delta)_p^sf(x)=|f(x)|^{p-2}f(x)\quad \mbox{ and }\quad \lim_{s\rightarrow 1^-}(-\Delta )_p^sf(x)=(-\Delta)_pf(x).
$$
Three representations of the fractional $p$-Laplacian were established in \cite[Theorems 2.1, 3.1, and 4.1]{TGV}.

In \cite{DJF} the logarithmic $p$-Laplace operator $\log (-\Delta )_p$ is defined as that operator for which the following expansion holds
$$
(-\Delta)_p^sf(x)=|f(x)|^{p-2}f(x)+s(\log (-\Delta)_p)f(x)+o(s),\quad s\rightarrow 0^+.
$$
A pointwise integral representation for $\log (-\Delta )_p$ was established in \cite[Theorem 1.1]{DJF}.

Banica, Gonz\'alez and S\'aez gave a definition of the fractional Laplacian on some noncompact manyfolds through an extension problem. They considered the real hyperbolic space $\mathbb H^n$, $n\geq 2$. This is the symplest example of Riemann symmetric spaces of noncompact type. In \cite{GSS} Gonz\'alez, S\'aez and Sire obtained layer solutions for the fractional Laplacian in hyperbolic settings. A reverse Faber-Krahn inequality for the $p$-Laplacian in hyperbolic spaces was proved in \cite[Theorem 1.4]{GV}. Motivated by the results in \cite{TGV}, J. Kim, M. Kim and Lee (\cite{KKL}) gave three equivalent definitions of the fractional $p$-Laplacian $(-\Delta_{\mathbb H^n})_p^s$, $s\in (0,1)$ and $p>1$, where $-\Delta_{\mathbb H^n}$ denotes the Laplacian in the space $\mathbb H^n$. 

Our objective in this paper is to define the logarithmic operator $\log (-\Delta_{\mathbb H^n})_p$ of the $p$-Laplacian $(-\Delta_{\mathbb H^n})_p$ with $p>1$ in the hyperbolic setting. We also obtain $\log(-\Delta_{\mathbb H^n})_p$ as the solution of an extension problem. 

There exist in the literature several models for the $n$-dimensional hyperbolic space $\mathbb H^n$. As in \cite{BGS} we consider $\mathbb H^n$ as the upper branch of the hyperboloid in $\mathbb R^{n+1}$ with the metric induced by the Lorentzian metric in $\mathbb R^{n+1}$ given by $-dx_0^2+dx_1^2+\cdots+dx_n^2$. More precisely, we consider
$$
\mathbb H^n=\big\{(x_0,x_1,\ldots, x_n)\in \mathbb R^{n+1}:\,x_0^2-x_1^2-\cdots -x_n^2=1,\,x_0>0\big\},
$$
and the metric $d_{\mathbb H^n}$ given by
$$
d_{\mathbb H^n}=d\rho ^2+(\sinh \rho )^2dw^2,
$$
where $dw^2$ defines the usual metric in the unit sphere $S^{n-1}$ in $\mathbb R^n$. The volume element is given by $(\sinh \rho )^{n-1}d\rho dw$. We denote  by $SO(1,n)$ the group of Lorentz transformations of $\mathbb R^{n+1}$ preserving the inner product $[\cdot,\cdot]$ defined by
$$
[x,z]=x_0z_0-x_1z_1-\cdots -x_nz_n,\quad x=(x_0,\cdots, x_n),\,z=(z_0,\cdots, z_n)\in \mathbb H^n.
$$
The hyperbolic space $\mathbb H^n$ is invariant under $SO(1,n)$.

The Laplace-Beltrami operator is defined by
$$
\Delta_{\mathbb H^n}=\frac{\partial ^2}{\partial \rho ^2}+(n-1)\frac{\cosh \rho }{\sinh \rho }\frac{\partial }{\partial \rho} +\frac{1}{(\sinh \rho )^2}\Delta_{S^{n-1}}.
$$
Let $s\in (0,1)$. The fractional $s$-power $(-\Delta_{\mathbb H^n})^s$ of $-\Delta_{\mathbb H^n}$ is given by
$$
(-\Delta_{\mathbb H^n})^sf(x)={\rm P.V.} \int_{\mathbb H^n}(f(z)-f(x))\mathcal K_{n,s}(d_{\mathbb H^n}(z,x))dz,
$$
when $f\in L^1(\mathbb H^n)\cap {\rm Lip}^\alpha (\mathbb H^n)$ for some $\alpha >2s$ (see \cite[Theorem 2.5]{BGS}). The integral kernel $\mathcal K_{n,s}$ is defined by (\cite[Definition 1.1]{KKL})
$$
\mathcal K_{n,s}(\rho )=sa_{n,s}\Big(-\frac{1}{\sinh \rho }\frac{\partial}{\partial \rho} \Big)^{\frac{n-1}{2}}\Big(\rho ^{-s-1/2}K_{s+1/2}\big(\frac{n-1}{2}\rho \big)\Big),\quad \rho >0,
$$
when $n\in \mathbb N$, $n\geq 3$ is odd, and
$$
\mathcal K_{n,s}(\rho )=\frac{sa_{n,s}}{\sqrt{\pi}}\int_\rho ^\infty \frac{\sinh r}{\sqrt{\cosh r-\cosh \rho }}\Big(-\frac{1}{\sinh r }\frac{\partial }{\partial r}\Big)^{\frac{n}{2}}\Big(r^{-s-1/2}K_{s+1/2}\big(\frac{n-1}{2}r\big)\Big)dr,\quad \rho >0,
$$
provided that $n\in \mathbb N$ is even. Here 
$$
a_{n,s}=\frac{(n-1)^{s+1/2}}{2^{\frac{n-1}{2}}\pi^{n/2}\Gamma (1-s)},\quad n\in \mathbb N,\;s\in (0,1),
$$
and $K_\nu$ is the modified Bessel function of second kind and order $\nu$.

In \cite{Ch} Chen obtained explicit expression for fractional and logarithmic operators on compact and noncompact complete Riemann manifolds. Note that the spectrum of the Laplace-Beltrami operator on compact Riemann manifolds is purely discrete in contrast with the noncompact case. The hyperbolic space, when modeled in the unit ball of $\mathbb R^n$, is a special case that has been studied in \cite[\S 4]{Ch}.

The fractional $p$-Laplacian on hyperbolic spaces has been studied in \cite{KKL}. Let $s\in (0,1)$ and $p>1$. The fractional $p$-Laplacian $(-\Delta_{\mathbb H^n})_p^s$ on $\mathbb H^n$ is defined by
$$
(-\Delta_{\mathbb H^n})_p^sf(x)={\rm P.V.} \int_{\mathbb H^n}\Phi _p(f(x)-f(y))\mathcal K_{n,s,p}(d_{\mathbb H^n}(x,y))dy,
$$
where $\Phi_p(z)=|z|^{p-2}z$, $z\in \mathbb R$. The integral kernel is given by
$$
\mathcal K_{n,s,p}(\rho )=sa_{n,s,p}\Big(-\frac{1}{\sinh \rho }\frac{\partial }{\partial \rho}\Big)^{\frac{n-1}{2}}\Big[\rho ^{-\frac{1+sp}{2}}K_{\frac{1+sp}{2}}\Big(\frac{n-1}{2}\rho \Big)\Big], \quad \rho >0,
$$
when $n\in \mathbb N$, $n\geq 3$ is odd, and
$$
\mathcal K_{n,s,p}(\rho )=\frac{sa_{n,s,p}}{\sqrt{\pi}}\int_\rho ^\infty \frac{\sinh r}{\sqrt{\cosh r-\cosh \rho }}\Big(-\frac{1}{\sinh r}\frac{\partial }{\partial r}\Big)^{\frac{n}{2}}\Big[r^{-\frac{1+sp}{2}}K_{\frac{1+sp}{2}}\Big(\frac{n-1}{2}r\Big)\Big]dr, \quad \rho >0,
$$
when $n\in \mathbb N$ is even, being
$$
a_{n,s,p}= \frac{p(n-1)^{\frac{1+sp}{2}}}{2^{s(p-2)+\frac{n+3}{2}}\pi ^{\frac{n-1}{2}}\Gamma (\frac{p+1}{2})\Gamma (1-s)}.
$$
Let us denote by $c_{n,p}$ the positive constant $c_{n,p}=a_{n,0,p}$. Note that, according to \cite[Corollary 3.3]{KKL} the kernel $\mathcal K_{n,s,p}$, $n\in \mathbb N$, $s\in (0,1)$, $p>1$,  is positive. Furthermore, by \cite[Proposition 1.2]{KKL} it holds that there exists $C>0$ such that, for each $s\in (0,1)$,
\begin{equation}\label{K0}
\frac{s\rho ^{-n-sp}}{C}\leq \mathcal K_{n,s,p}(\rho )\leq Cs\rho ^{-n-sp},\quad 0<\rho <1,
\end{equation}
and 
\begin{equation}\label{Kinfty}
\frac{s\rho ^{-1-\frac{sp}{2}}e^{-(n-1)\rho }}{C}\leq \mathcal K_{n,s,p}(\rho )\leq Cs\rho ^{-1-\frac{sp}{2}}e^{-(n-1)\rho },\quad\rho \geq 1.
\end{equation}

We consider the Poisson type kernel defined by
$$
P_n(\rho ,t)=\Big(-\frac{1}{\sinh \rho }\frac{\partial }{\partial \rho}\Big)^{\frac{n-1}{2}}\Big(\mathbb K_n(\sqrt{\rho ^2+t^2})\Big),\quad \rho >0,\,t\geq 0,
$$
when $n\in \mathbb N$, $n\geq 3$ is odd, and
$$
P_n(\rho ,t)=\int_\rho ^\infty \frac{\sinh r}{\sqrt{\cosh r-\cosh\rho }}\Big(-\frac{1}{\sinh r}\frac{\partial }{\partial r}\Big)^{\frac{n}{2}}\Big(\mathbb K_n(\sqrt{r^2+t^2})\Big)dr,\quad \rho >0,\,t\geq 0,
$$
when $n\in \mathbb N$, $n$ is even, where
$$
\mathbb K_n(u)=u^{-1/2}K_{\frac{1}{2}}\big(\frac{n-1}{2}u\big)=\Big(\frac{\pi}{n-1}\Big)^{1/2}\frac{e^{-\frac{n-1}{2}u}}{u},\quad u>0.
$$

Note that
\begin{equation}\label{limKP}
\lim_{s\rightarrow 0^+}\frac{\mathcal K_{n,s,p}(\rho )}{s}=\sigma_{n,p}P_n(\rho ,0),\quad \rho >0,
\end{equation}
where $\sigma_{n,p}=c_{n,p}$, when $n\in \mathbb N$ is odd, $n\geq 3$, and $\sigma_{n,p}=\pi ^{-1/2}c_{n,p}$, when $n\in \mathbb N$ is even.

The fractional $p$-Laplacian $(-\Delta_{\mathbb H^n})_p^s$ is also defined by using the heat semigroup associated with $\Delta_{\mathbb H^n}$ and as the solution of an extension problem in \cite{KKL}. These results can be seen as a version for $\mathbb H^n$ of those ones proved in \cite{TGV} for the Euclidean Laplacian.

In \cite[Theorem 1.5]{KKL} it is proved that if $n\in \mathbb N$, $n\geq 2$, and $p\geq 2$, then
$$
\lim_{s\rightarrow 1^-}(-\Delta_{\mathbb H^n})_p^sf(x)=(-\Delta_{\mathbb H^n})_pf(x),
$$
provided that $f\in C_b^2(\mathbb H^n)$ and $x\in \mathbb H^n$ such that $\nabla f(x)\not=0$. Here $C_b^2(\mathbb H^n)$ denotes the space that consists of all those $C^2$ and bounded functions in $\mathbb H^n$.

In order to define the logarithmic operator $\log (-\Delta_{\mathbb H^n})_p$ of $(-\Delta_{\mathbb H^n})_p$ we need to establish the following result. Let us denote by $C_c(\mathbb H^n)$ the space of complex functions continuous in $\mathbb H^n$ with compact support, and {by ${\rm Lip}_{\rm loc}^\alpha (\mathbb H^n)$, with $\alpha \in (0,1)$, the space of complex functions defined on $\mathbb H^n$ such that, for every $x\in \mathbb H^n$, and $r>0$,
$$
\sup_{y\in \mathbb H^n, d_{\mathbb H^n}(x,y)\leq r}\frac{|f(x)-f(y)|}{d_{\mathbb H^n}(x,y)^\alpha}<\infty.
$$
\begin{thm}\label{Th1.1}
Let $n\in \mathbb N$, $n\geq 2$, $1<p<\infty$, and $\alpha\in (0,1)$. There exists $A_{n,p}>0$ such that, for every $f\in C_c(\mathbb H^n)\cap {\rm Lip}_{\rm loc}^\alpha (\mathbb H^n)$,
$$
\lim_{s\rightarrow 0^+}(-\Delta_{\mathbb H^n})_p^sf(x)=A_{n,p}\Phi _p(f(x)),\quad x\in \mathbb H^n.
$$
\end{thm}
\vspace*{0.2cm}
We define the logarithm operator $\log(-\Delta_{\mathbb H^n})_p$ as follows
$$
(\log(-\Delta_{\mathbb H^n})_p)f(x)=\lim_{s\rightarrow 0^+}\frac{(-\Delta_{\mathbb H^n})_p^sf(x)-A_{n,p}\Phi _p(f(x))}{s},\quad x\in \mathbb H^n.
$$
We establish pointwise representations for the logarithm operator. 

\begin{thm}\label{Th1.2}
Let $n\in \mathbb N$, $n\geq 2$, $\alpha \in (0,1)$, $1<p<\infty$, and $f\in C_c (\mathbb H^n)\cap {\rm Lip}_{\rm loc}^\alpha(\mathbb H^n)$.

(a) There exist $\alpha _{n,p},\beta_{n,p}\in \mathbb R$ such that if $x\in \mathbb H^n$ and $\supp f\subseteq B(x,R)$, with $R\geq 1$, then,
    \begin{align*}
    (\log(-\Delta_{\mathbb H^n})_p)f(x)&=(\alpha _{n,p}+\beta_{n,p}\log R)\Phi _p(f(x))\\
    &\quad +\sigma_{n,p}\int_{B(x,R)}\Phi _p(f(x)-f(y))P_n(d_{\mathbb H^n}(x,y),0)dy.
\end{align*}
   
(b) There exists $\alpha_{n,p}\in \mathbb R$ such that
\begin{align*}
    (\log(-\Delta_{\mathbb H^n})_p)f(x)&=\alpha _{n,p}\Phi _p(f(x))\\
    &\hspace{-1cm}+\sigma_{n,p}\left(\int_{B(x,1)}\Phi _p(f(x)-f(y))P_n(d_{\mathbb H^n}(x,y),0)dy\right.\\
    &\hspace{-1cm}+\left.\int_{\mathbb H^n\setminus B(x,1)}\big(\Phi _p(f(x)-f(y))-\Phi _p(f(x))\big)P_n(d_{\mathbb H^n}(x,y),0)dy\right), \quad x\in \mathbb H^n.
\end{align*}
Here $\sigma _{n,p}$ is the constant given in \eqref{limKP}.
\end{thm}

The logarithm operator $\log(-\Delta_{\mathbb H^n})_p$ can be characterized as the solution of an extension pro\-blem. We define the extension operator
$$
E_n(f)(x,t)=\int_{\mathbb H^n}f(y)P_n(d_{\mathbb H^n}(x,y),t)dy,\quad x\in \mathbb H^n\mbox{ and }t>0.
$$
Note that for suitable functions $f$ we have that
$$
\Delta_{\mathbb H^n}E_n(f)+\frac{1}{t}\frac{\partial}{\partial t}E_n(f)+\frac{\partial ^2}{\partial t^2}E_n(f)=0.
$$
\begin{thm}\label{Th1.3}
Let $n\in \mathbb N$, $n\geq 2$, $\alpha \in (0,1)$, $1<p<\infty$, and $f\in C_c(\mathbb H^n)\cap {\rm Lip}^\alpha _{\rm loc}(\mathbb H^n)$. There exist $\gamma_n$ and $\delta_{n,p}\in \mathbb R$ such that
\begin{align*}
(\log(-\Delta_{\mathbb H^n})_p)f(x)&=\lim_{t\rightarrow 0^+}\sigma_{n,p}\Big(E_n\big(\Phi _p(f(x)-f(\cdot))-\Phi_p(f(x))\big)+\gamma_n\Phi _p(f(x))\log t\Big)\\
&\quad +\delta_{n,p}\Phi _p(f(x)),\quad x\in \mathbb H^n,
\end{align*}
being $\sigma_{n,p}$ the constant in \eqref{limKP}.
\end{thm}

As it was mentioned, the logarithmic $p$-Laplacian in $\mathbb R^n$, $\log(-\Delta)_p$, was defined in \cite{DJF} and a pointwise integral representation for $\log (-\Delta)_p$ was established (\cite[Theorem 1.1]{DJF}). As far as we know the operator $\log(-\Delta )_p$ has not been obtained as solution of an extension problem. By using the arguments developed in the proof of Theorem \ref{Th1.3} we can complete the results in \cite{DJF} characterizing $\log (-\Delta)_p$ as the solution of an extension problem.

As in \cite[Theorem 1.2]{CHW} we consider the extension operator $\mathbb E_n$ defined by
$$
\mathbb E_n(f)(x,t)=\frac{\Gamma (n/2)}{2\pi ^{n/2}}\int_{\mathbb R^n}\frac{f(z)}{(|x-z|^2+t^2)^{n/2}}dz,\quad x\in \mathbb R^n\mbox{ and }t>0.
$$
For $k\in \mathbb N\cup\{0\}$, $\sigma \in (0,1)$, and $\alpha =k+\sigma$, we denote by $C_c^\alpha(\mathbb R^n)$, the space of $k$-times continuously differentiable functions with compact support such that the derivatives up to order $k$ are $\sigma$-Lipschitz in $\mathbb R^n$.
\begin{thm}\label{Th1.4}
    Let $1<p<\infty$, and $f\in C_c^\alpha (\mathbb R^n)$ for some $\alpha >0$. Then, for a certain $\alpha_{n,p}\in \mathbb R$,
    \begin{align*}
    (\log (-\Delta)_p)f(x)&\\
    &\hspace{-2cm}=\lim_{t\rightarrow 0^+}p\Big(\mathbb E_n\big(\Phi _p(f(x)-f(\cdot ))-\Phi_p(f(x))\big)-\Phi _p(f(x))\log t\Big)+\alpha_{n,p}\Phi _p(f(x)),\quad x\in \mathbb R^n.
    \end{align*}
\end{thm}

The paper is organized as follows. In the next section we present several auxiliary results needed for the proofs of our main theorems, and then prove them in subsequent sections.

Throughout this paper, $C$ and $c$ represent positive constants whose values may vary from one ocurrence to another.

\section{Auxiliary results}\label{S2}

The following result plays a key role in the proofs given in the next sections.

\begin{lem}\label{key}
For every $m\in \mathbb N$ we have 
\begin{equation}\label{formderiv}
\Big(\frac{1}{\sinh \rho }\frac{\partial }{\partial\rho }\Big)^m=\sum_{k=1}^m\frac{T_{m,k}(\rho )}{(\sinh \rho)^{2m-k}}\Big(\frac{1}{\rho }\frac{\partial }{\partial\rho }\Big)^k,\quad \rho >0,
\end{equation}
where, for each $k=1,\ldots, m$, 
\begin{equation}\label{Tmk}
T_{m,k}(\rho )=\sum_{j=0}^k Q_j^{m,k}(\rho )\rho ^j,\quad \rho >0,
\end{equation}
with $Q_j^{m,k}(\rho )=R_j^{m,k}(\sinh \rho ,\cosh \rho)$, $\rho >0$, where $R_j^{m,k}$ is a homogeneous polynomial of degree $m-k$, for $j=0,\ldots , k$. In particular, $T_{m,m}(\rho )=\rho  ^m$, $\rho >0$.

Moreover, for every $m\in \mathbb N$, and $k=1,\ldots, m$, the following properties hold:

(i) For certain $(b_i^{m,k})_{i=0}^\infty\subseteq \mathbb R$, 
$$
T_{m,k}(\rho )=\sum_{i=0}^\infty b_i^{m,k}\rho ^{2m-k+2i},\quad \rho >0,
$$
and the series is absolutely convergent for each $\rho >0$. Note that if $k\in \mathbb N$ is even (respectively, odd) $T_{m,k}$ can be extended as an even function (respectively, odd function).

(ii) For each $j=0,\ldots,k$, there exists $c_j^{m,k}\in \mathbb R$ such that
$$
\Big|\frac{Q_j^{m,k}(\rho )}{(\sinh \rho )^{2m-k}}-c_j^{m,k}e^{-m\rho }\Big|\leq Ce^{-(m+2)\rho},\quad \rho \geq 1,
$$

where $C$ does not depend on $j$ or $k$. Furthermore, $|c_k^{m,k}|=(-1)^{m+k}c_k^{m,k}$.

(iii) For each $j=0,\ldots,k$, 
$$
\Big|\frac{Q_j^{m,k}(\rho )}{(\sinh \rho )^{2m-k}}\Big|\leq Ce^{-m\rho },\quad \rho \geq 1,
$$
for certain $C>0$ that does not depend on $k$ or $j$.
\end{lem}
\begin{proof}
We proceed by induction on $m\in \mathbb N$ to establish the formula \eqref{formderiv}. When $m=1$, it is clear that property holds by taking $T_{1,1}(\rho )=\rho $, $\rho >0$.
Suppose that \eqref{formderiv} is satisfied for a certain $m\in \mathbb N$ and that $T_{m,k}$, $k=1,\ldots, m$, is given by \eqref{Tmk}.

Then, we can write
\begin{align*}
 \Big(\frac{1}{\sinh \rho }\frac{\partial}{\partial \rho}\Big)^{m+1}&=\sum_{k=1}^m\left[\frac{\rho T_{m,k}(\rho )}{(\sinh \rho)^{2m+1-k}}\Big(\frac{1}{\rho }\frac{\partial }{\partial \rho} \Big)^{k+1}+ \frac{T_{m,k}'(\rho )}{(\sinh \rho)^{2m+1-k}}\Big(\frac{1}{\rho }\frac{\partial}{\partial \rho }\Big)^k\right.\\
 &\quad \left.-(2m-k)\frac{T_{m,k}(\rho )\cosh \rho }{(\sinh \rho)^{2m+2-k}}\Big(\frac{1}{\rho }\frac{\partial }{\partial \rho}\Big)^k\right]\\
 &=\sum_{k=2}^{m+1}\frac{\rho T_{m,k-1}(\rho )}{(\sinh \rho)^{2m+2-k}}\Big(\frac{1}{\rho }\frac{\partial }{\partial \rho}\Big)^k\\
 &\quad + \sum_{k=1}^m\frac{T_{m,k}'(\rho )\sinh \rho -(2m-k)T_{m,k}(\rho )\cosh \rho }{(\sinh \rho)^{2m+2-k}}\Big(\frac{1}{\rho}\frac{\partial}{\partial \rho}\Big)^k\\
 &=\sum_{k=1}^{m+1}\frac{T_{m+1,k}(\rho )}{(\sinh \rho )^{2(m+1)-k}}\Big(\frac{1}{\rho}\frac{\partial}{\partial \rho}\Big)^k,
\end{align*}
where
\begin{align}\label{T1}
&T_{m+1,1}(\rho )=T_{m,1}'(\rho )\sinh \rho -(2m-1)T_{m,1}(\rho )\cosh \rho,\quad \rho >0,\\
&T_{m+1,m+1}(\rho )=\rho T_{m,m}(\rho)=\rho ^{m+1},\quad \rho >0,\label{Tm}
\end{align}
and, for $k=2,\ldots, m$, 
\begin{equation}\label{Tk}
T_{m+1,k}(\rho )=\rho T_{m,k-1}(\rho )+T_{m,k}'(\rho )\sinh \rho -(2m-k)T_{m,k}(\rho )\cosh \rho,\quad \rho >0.
\end{equation}
Let $k=1,\ldots,m$. From \eqref{Tmk} we get
$$
T_{m,k}'(\rho )=(Q_k^{m,k})'(\rho )\rho ^k+\sum_{j=0}^{k-1}\big((Q_j^{m,k})'(\rho )+(j+1)Q_{j+1}^{m,k}(\rho )\big)\rho ^j,\quad \rho >0.
$$
Then, $T_{m+1,k}(\rho)=\sum_{j=0}^kQ_j^{m+1,k}(\rho )\rho ^j$, $\rho >0$, $k=1,\ldots, m+1$, where, for each $\rho >0$,
\begin{align}\label{Q11}
Q_j^{m+1,m+1}(\rho)&=\left\{
\begin{array}{ll}
0,&j=0,\ldots, m,\\[0.2cm]
1,&j=m+1,
\end{array}
\right.\nonumber\\
 Q_j^{m+1,1}(\rho)&=\left\{
\begin{array}{ll}
((Q_0^{m,1})'(\rho)+Q_1^{m,1}(\rho ))\sinh \rho-(2m-1)Q_0^{m,1}(\rho )\cosh \rho,&j=0,\\[0.2cm]
(Q_1^{m,1})'(\rho )\sinh \rho-(2m-1)Q_1^{m,1}(\rho )\cosh \rho,&j=1,
\end{array}
\right.
\end{align}
and, for $k=2,\ldots, m$,
\begin{equation}\label{Qkk}
Q_j^{m+1,k}(\rho)=\left\{
\begin{array}{ll}
((Q_0^{m,k})'(\rho)+Q_1^{m,k}(\rho ))\sinh \rho-(2m-k)Q_0^{m,k}(\rho )\cosh \rho,&j=0,\\[0.2cm]
Q_{k-1}^{m,k-1}(\rho )+(Q_k^{m,k})'(\rho )\sinh \rho-(2m-k)Q_k^{m,k}(\rho )\cosh \rho,&j=k,
\end{array}
\right.
\end{equation}
and, if $j=1,\ldots, k-1$, 
$$
Q_j^{m+1,k}(\rho)=Q_{j-1}^{m,k-1}(\rho)+((Q_j^{m,k})'(\rho)+(j+1)Q_{j+1}^{m,k}(\rho ))\sinh \rho-(2m-k)Q_j^{m,k}(\rho )\cosh \rho.
$$

Since for each $j=0,\ldots, k$, we have that $Q_j^{m,k}(\rho)=R_j^{m,k}(\sinh \rho ,\cosh \rho )$, $\rho >0$, where $R_j^{m,k}$ is a homogeneous polynomial of degreee $m-k$, and, for every $r=0,\ldots, m-k$, and $\rho >0$,
$$
\frac{d}{d\rho }\big [(\sinh \rho )^r(\cosh \rho )^{m-k-r}\big]=r(\sinh \rho )^{r-1}(\cosh \rho )^{m+1-k-r}+(m-k-r)(\sinh \rho )^{r+1}(\cosh \rho )^{m-k-r-1},
$$
it is clear that $(Q_j^{m,k})'\sinh \rho$, $Q_j^{m,k}\sinh \rho$, and $Q_j^{m,k}\cosh \rho$, $j=0,\ldots, k$, can be represented by means of homogeneous polynomials of degree $m+1-k$. Also, $R_j^{m,k-1}$, $j=0,\ldots, k-1$, and $k=2,\ldots, m$, are homogeneous polynomials of degree $m-(k-1)=m+1-k$. By using these arguments and from \eqref{Q11} and \eqref{Qkk} we conclude that for $k=1,\ldots, m+1$, and $j=0,\ldots, k$, $Q_j^{m+1, k}(\rho )=R_j^{m+1,k}(\sinh \rho ,\cosh \rho )$, $\rho >0$, with $R_j^{m+1,k}$ a homogeneous polynomial of degree $m+1-k$. 

Let us next prove $(i)$. We again proceed by induction on $m\in \mathbb N$, and use \eqref{T1}-\eqref{Tk}. It is clear for $m=1$. Suppose that $(i)$ holds for certain $m\in \mathbb N$. Note that, since $T_{m+1,m+1}(\rho )=\rho ^{m+1}$, $\rho >0$, it follows that $b_0^{m+1,m+1}=1$, and $b_i^{m+1,m+1}=0$, $i\in \mathbb N$. On the other hand by \eqref{T1} we get
\begin{align*}
    T_{m+1,1}(\rho )&=\sinh \rho \sum_{i=0}^\infty b_i^{m,1}(2m-1+2i)\rho ^{2m-2+2i}-(2m-1)\cosh \rho \sum_{i=0}^\infty b_i^{m,1}\rho ^{2m-1+2i}\\
    &=\sum_{i,r=0}^\infty b_i^{m,1}\frac{2m-1+2i}{(2r+1)!}\rho ^{2m-1+2(i+r)}-\sum_{i,r=0}^\infty b_i^{m,1}\frac{2m-1}{(2r)!}\rho ^{2m-1+2(i+r)}\\
    &=(2m-1)b_0^{m,1}\sum_{r=1}^\infty \Big(\frac{1}{(2r+1)!}-\frac{1}{(2r)!}\Big)\rho ^{2m-1+2r}\\
    &\quad +\sum_{i=1}^\infty b_i^{m,1}\sum_{r=0}^\infty \Big(\frac{2m-1+2i}{(2r+1)!}-\frac{2m-1}{(2r)!}\Big)\rho ^{2m-1+2(i+r)}\\
    &=(2m-1)b_0^{m,1}\sum_{r=0}^\infty \Big(\frac{1}{(2r+3)!}-\frac{1}{(2r+2)!}\Big)\rho ^{2m+1+2r}\\
    &\quad +\sum_{i=0}^\infty b_{i+1}^{m,1}\sum_{r=0}^\infty \Big(\frac{2m+1+2i}{(2r+1)!}-\frac{2m-1}{(2r)!}\Big)\rho ^{2m+1+2(i+r)}\\
     &=(2m-1)b_0^{m,1}\sum_{r=0}^\infty \Big(\frac{1}{(2r+3)!}-\frac{1}{(2r+2)!}\Big)\rho ^{2m+1+2r}\\
    &\quad +\sum_{r=0}^\infty \sum_{i=r}^\infty b_{i-r+1}^{m,1}\Big(\frac{2m+1+2(i-r)}{(2r+1)!}-\frac{2m-1}{(2r)!}\Big)\rho ^{2m+1+2i}\\
     &=(2m-1)b_0^{m,1}\sum_{r=0}^\infty \Big(\frac{1}{(2r+3)!}-\frac{1}{(2r+2)!}\Big)\rho ^{2m+1+2r}\\
    &\quad +\sum_{i=0}^\infty \sum_{r=0}^i b_{i-r+1}^{m,1}\Big(\frac{2m+1+2(i-r)}{(2r+1)!}-\frac{2m-1}{(2r)!}\Big)\rho ^{2m+1+2i}\\
    &=\sum_{i=0}^\infty b_i^{m+1,1}\rho ^{2(m+1)-1+2i},\quad \rho >0,
\end{align*}
for certain $(b_i^{m+1,1})_{i=0}^\infty \subseteq \mathbb R$, being the series absolutely convergent for each $\rho >0$. Analogously, when $k=2,\ldots, m$, by using \eqref{Tk} we can write
\begin{align*}
T_{m+1,k}(\rho )&=\sum_{i=0}^\infty b_i^{m,k-1}\rho ^{2m-(k-1)+2i+1}+\sinh \rho \sum_{i=0}^\infty b_i^{m,k}(2m-k+2i)\rho ^{2m-k+2i-1}\\
&\quad -(2m-k)\cosh \rho \sum_{i=0}^\infty b_i^{m,k}\rho ^{2m-k+2i} \\
&=\sum_{i=0}^\infty b_i^{m,k-1}\rho ^{2(m+1)-k+2i}+\sum_{i,r=0}^\infty b_i^{m,k}\Big(\frac{2m-k+2i}{(2r+1)!}-\frac{2m-k}{(2r)!}\Big)\rho ^{2m-k+2(i+r)}\\
&=\sum_{i=0}^\infty b_i^{m,k-1}\rho ^{2(m+1)-k+2i}+b_0^{m,k}(2m-k)\sum_{r=1}^\infty  \Big(\frac{1}{(2r+1)!}-\frac{1}{(2r)!}\Big)\rho ^{2m-k+2r}\\
&\quad +\sum_{i=1}^\infty\sum_{r=0}^\infty b_i^{m,k}\Big(\frac{2m-k+2i}{(2r+1)!}-\frac{2m-k}{(2r)!}\Big)\rho ^{2m-k+2(i+r)}\\
&=\sum_{i=0}^\infty b_i^{m,k-1}\rho ^{2(m+1)-k+2i}+b_0^{m,k}(2m-k)\sum_{r=0}^\infty  \Big(\frac{1}{(2r+3)!}-\frac{1}{(2r+2)!}\Big)\rho ^{2m-k+2r+2}\\
&\quad +\sum_{i=0}^\infty\sum_{r=0}^\infty b_{i+1}^{m,k}\Big(\frac{2m-k+2i+2}{(2r+1)!}-\frac{2m-k}{(2r)!}\Big)\rho ^{2m-k+2(i+r)+2}\\
&=\sum_{i=0}^\infty b_i^{m,k-1}\rho ^{2(m+1)-k+2i}+b_0^{m,k}(2m-k)\sum_{r=0}^\infty  \Big(\frac{1}{(2r+3)!}-\frac{1}{(2r+2)!}\Big)\rho ^{2(m+1)-k+2r}\\
&\quad +\sum_{i=0}^\infty\sum_{r=0}^ib_{i-r+1}^{m,k}\Big(\frac{2m-k+2(i-r)+2}{(2r+1)!}-\frac{2m-k}{(2r)!}\Big)\rho ^{2(m+1)-k+2i}\\
&=\sum_{i=0}^\infty b_i^{m+1,k}\rho ^{2(m+1)-k+2i},\quad \rho >0,
\end{align*}
for certain $(b_i^{m+1,k})_{i=0}^\infty\subseteq \mathbb R$, for which  the series is absolutely convergent, for every $\rho >0$. Hence the proof of $(i)$ is finished.

To show $(ii)$ we fix $m\in \mathbb N$, $k=1,\ldots, m$, and $j=0,\ldots, k$, and  write
$$
Q_j^{m,k}(\rho )=\sum_{r=0}^{m-k}a_r^{m,k,j}(\sinh \rho )^{m-k-r}(\cosh \rho )^{r},\quad \rho >0,
$$
with $a_r^{m,k,j}\in \mathbb R$, $r=0,\ldots, m-k$. We observe that, for each $r=0,\ldots, m-k$,
$$
\frac{(\sinh \rho )^{m-k-r}(\cosh \rho )^r}{(\sinh \rho )^{2m-k}}=\frac{(\cosh \rho )^r}{(\sinh \rho )^{m+r}}=2^me^{-m\rho }\frac{(1+e^{-2\rho })^r}{(1-e^{-2\rho })^{m+r}},\quad \rho >0.
$$
Then, if $c_j^{m,k}=2^m\sum_{r=0}^{m-k}a_r^{m,k,j}$ we have that
$$
\frac{Q_j^{m,k}(\rho )}{(\sinh \rho )^{2m-k}}-c_j^{m,k}e^{-m\rho }=2^me^{-m\rho }\sum_{r=0}^{m-k}a_r^{m,k,j}\Big(\frac{(1+e^{-2\rho })^r}{(1-e^{-2\rho })^{m+r}}-1\Big),\quad \rho >0,
$$
and thus, when $\rho \geq 1$,
\begin{align*}
 \Big|\frac{Q_j^{m,k}(\rho )}{(\sinh \rho )^{2m-k}}-c_j^{m,k}e^{-m\rho }\Big|&\leq Ce^{-m\rho }\sum_{r=0}^{m-k}\frac{(1+e^{-2\rho })^r-(1-e^{-2\rho })^{m+r}}{(1-e^{-2\rho })^{m+r}}\\
 &\hspace{-3cm}\leq  Ce^{-m\rho }\sum_{r=0}^{m-k}\big((1+e^{-2\rho })^r-(1-e^{-2\rho })^{m+r}\big)\\
  &\hspace{-3cm}=Ce^{-m\rho }\sum_{r=0}^{m-k}\Big(\sum_{\ell =1}^r\binom{r}{\ell}e^{-2\ell\rho}-\sum_{\ell =1}^{m+r}(-1)^\ell \binom{m+r}{\ell}e^{-2\ell\rho }\Big)\leq Ce^{-(m+2)\rho },
\end{align*}
for certain $C>0$ that can depend on $m$ but not on $j$ or $k$. Thus,
we conclude that there exists $C>0$ independent of $j$ and $k$ such that
$$
 \Big|\frac{Q_j^{m,k}(\rho )}{(\sinh \rho )^{2m-k}}-c_j^{m,k}e^{-m\rho }\Big|\leq Ce^{-(m+2)\rho },\quad \rho \geq 1.
$$
Note that from this estimation property in $(iii)$ follows easily. 

Finally, let us show that, for every $m\in \mathbb N$, and $k=1,\ldots, m$,
\begin{equation}\label{cmk}
|c_k^{m,k}|=(-1)^{m+k}c_k^{m,k}.
\end{equation}
Recall that $c^{m,k}_k=2^m\sum_{r=0}^{m-k}a_r^{m,k,k}$. We proceed again by induction on $m$.  

Since $T_{1,1}(\rho )=\rho $, we have that $Q_1^{1,1}(\rho )=1$. Thus, $c_1^{1,1}=2$, so the property is verified for $m=1$. Moreover, it is obvious from \eqref{Tm} that, for each $m\in \mathbb N$, $Q^{m,m}_m(\rho )=1$, $\rho >0$. Hence, $c_m^{m,m}=2^m$ and \eqref{cmk} is trivially satisfied.

Suppose now that \eqref{cmk} holds for certain $m\in \mathbb N$, and each $k=1,\ldots, m$. From \eqref{Q11} we get
\begin{align*}
Q_1^{m+1,1}(\rho )&=\sum_{r=0}^{m-2}a_r^{m,1,1}(m-1-r)(\sinh \rho )^{m-1-r}(\cosh \rho )^{r+1}\\
&\quad +\sum_{r=1}^{m-1}ra_r^{m,1,1}(\sinh \rho )^{m+1-r}(\cosh \rho )^{r-1}\\
&\quad -(2m-1)\sum_{r=0}^{m-1}a_r^{m,1,1}(\sinh \rho )^{m-1-r}(\cosh \rho )^{r+1},\quad \rho >0.
\end{align*}
Thus, 
\begin{align*}
    2^{-(m+1)}c_1^{m+1,1}&=\sum_{r=0}^{m-2}a_r^{m,1,1}(m-1-r)+\sum_{r=1}^{m-1}ra_r^{m,1,1}-(2m-1)\sum_{r=0}^{m-1}a_r^{m,1,1}\\
    &=-m\sum_{r=0}^{m-1}a_r^{m,1,1}=-2^{-m}mc_1^{m,1},
\end{align*}
and it follows that $|c_1^{m+1,1}|=2m|c_1^{m,1}|=(-1)^{m+1}2mc_1^{m,1}=(-1)^{m+2}c_1^{m+1,1}$.

On the other hand, according to \eqref{Qkk}, for $k=2,\ldots, m$, we have that
\begin{align*}
Q_k^{m+1,k}(\rho )&=\sum_{r=0}^{m-k+1}a_r^{m,k-1,k-1}(\sinh \rho )^{m-k+1-r}(\cosh \rho )^r\\
&\quad +\sum_{r=0}^{m-k-1}a_r^{m,k,k}(m-k-r)(\sinh \rho )^{m-k-r}(\cosh \rho )^{r+1}\\
&\quad +\sum_{r=1}^{m-k}ra_r^{m,k,k}(\sinh \rho )^{m-k-r+2}(\cosh \rho )^{r-1}\\
&\quad -(2m-k)\sum_{r=0}^{m-k}a_r^{m,k,k}(\sinh \rho )^{m-k-r}(\cosh \rho )^{r+1},\quad \rho >0.
\end{align*}
Then, for every $k=2,\ldots, m$,
\begin{align*}
    2^{-(m+1)}c_k^{m+1,k}&\\
    &\hspace{-1cm}=\sum_{r=0}^{m-k+1}a_r^{m,k-1,k-1}+\sum_{r=0}^{m-k-1}a_r^{m,k,k}(m-k-r)+\sum_{r=1}^{m-k}ra_r^{m,k,k}-(2m-k)\sum_{r=0}^{m-k}a_r^{m,k,k}\\
    &\hspace{-1cm}=2^{-m}c_{k-1}^{m,k-1}-m\sum_{r=0}^{m-k}a_r^{m,k,k}=2^{-m}c_{k-1}^{m,k-1}-2^{-m}mc_k^{m,k}\\\
    &\hspace{-1cm}=(-1)^{m+k+1}2^{-m}(|c_{k-1}^{m,k-1}|+m|c_k^{m,k}|),
\end{align*}
that is, $|c_k^{m+1,k}|=2(|c_k^{m,k-1}|+m|c_k^{m,k}|)=(-1)^{m+k+1}c_k^{m+1,k}$. 
\end{proof}
The next two results will be used in the proof of Theorem \ref{Th1.3}.
\begin{lem}\label{FadiBruno}
    Let $k\in \mathbb N$, and $\alpha,\beta>0$. If $g(u)=e^{-\alpha\sqrt{u^2+\beta}}(u^2+\beta)^{-1/2}$, then
    $$
    \Big(\frac{1}{u}\frac{d}{du}\Big)^kg(u)=g(u)\sum_{(s_1,\ldots, s_k,\ell)\in \mathcal I_k}\frac{b_{s_1,\ldots, s_k,\ell}^{(\alpha)}}{(u^2+\beta)^{k-\frac{1}{2}(\sum_{j=1}^ks_j-\ell)}},\quad u>0,
    $$
    for certain $b_{s_1,\ldots, s_k,\ell}^{(\alpha)}\in \mathbb R$, $(s_1,\ldots, s_k,\ell )\in \mathcal I_k$, being $\mathcal I_k=\{(s_1,\ldots, s_k,\ell)\in (\mathbb N\cup \{0\})^{k+1}: \,\sum_{j=1}^kjs_j=k, \mbox{ \it and }\ell =0,\ldots, s_1+\cdots+s_k\}$. 
\end{lem}
\begin{proof}
First observe that if $f$ is a $k$-times derivable function in $(0,\infty)$ then 
$$
\Big(\frac{1}{u}\frac{d}{du}\Big)^k[f(u^2)]=2^k\frac{d^k}{ds^k}f(s)_{|s=u^2},\quad u>0,
$$ 
and by virtue of Fa\`a di Brunos's formula, 
$$
\frac{d^k}{dv^k}[f(\sqrt{v+\beta})]=\sum_{\substack{(s_1,\ldots, s_k)\in (\mathbb N\cup\{0\})^k\\s_1+2s_2+\cdots+ks_k=k}}\frac{k!}{s_1!\cdots s_k!1!^{s_1}\cdots k!^{s_k}}f^{(s_1+\cdots +s_k)}(\sqrt{v+\beta})\prod_{j=1}^k\frac{b_j}{(v+\beta)^{(j-1/2)s_j}},
$$
where $b_j\in \mathbb R$, $j=1,\ldots, k$.

Since $g(u)=f(\sqrt{u^2+\beta})$, where $f(u)=\frac{e^{-\alpha u}}{u}$, $u>0$, we get 
\begin{align*}
    \Big(\frac{1}{u}\frac{d}{du}\Big)^kg(u)&=2^k\sum_{\substack{(s_1,\ldots,s_k)\in (\mathbb N\cup\{0\})^k\\s_1+2s_2+\cdots+ks_k=k}}\frac{c_{s_1,\ldots, s_k}}{(u^2+\beta)^{k-\frac{1}{2}\sum_{j=1}^ks_j}}\\
    &\quad \times \sum_{\ell =0}^{s_1+\cdots+s_k}a_{s_1,\ldots,s_k,\ell}^{(\alpha)} e^{-\alpha\sqrt{u^2+\beta}}(u^2+\beta)^{-\frac{1}{2}(\ell +1)}\\
    &=\sum_{(s_1,\ldots, s_k,\ell )\in \mathcal I_k}b_{s_1,\ldots, s_k,\ell}^{(\alpha)}\frac{e^{-\alpha\sqrt{u^2+\beta}}}{(u^2+\beta)^{k+\frac{1}{2}-\frac{1}{2}(\sum_{j=1}^ks_j-\ell)}},\quad u>0,
\end{align*}
and the result is obtained. Here $c_{s_1,\ldots, s_k}$, $a_{s_1,\ldots, s_k,\ell}^{(\alpha)}$, and $b_{s_1,\ldots, s_k,\ell}^{(\alpha)}$ are suitable real numbers, for each  $(s_1,\ldots, s_k,\ell )\in \mathcal I_k$. 
\end{proof}
\begin{lem}\label{psivarphi}
Let $m\in \mathbb N$. Define the functions $\psi$ and $\varphi$  defined by
$$
\psi (t)=\int_t^1\frac{\rho^{2m}}{(\rho^2+t^2)^{m+1/2}}d\rho\quad \mbox{ and }\quad \varphi (t)=\int_t^1\int_\rho ^\infty \frac{r\rho ^{2m-1}}{\sqrt{r^2-\rho ^2}(r^2+t^2)^{m+1/2}}drd\rho,\quad t\in (0,1).
$$
Then,
\begin{equation}\label{psi}
\lim_{t\rightarrow 0^+}\big(\psi (t)+\log t\big)=\alpha _m,
\end{equation}
and 
\begin{equation}\label{varphi}
\lim_{t\rightarrow 0^+}\big(\varphi (t)+\sigma _m\log t\big)=\beta _m
\end{equation}
where, 
$$
\alpha _m=\int_0^1\Big(\frac{1}{(1+u^2)^{m+1/2}}-1\Big)\frac{du}{u},\quad \beta_m=\int_0^1\frac{H_m(u)-H_m(0)}{u}du, \quad \sigma_m=H_m(0),
$$
and
$$
H_m(u)=\int_1^\infty \frac{s}{\sqrt{s^2-1}(s^2+u^2)^{m+1/2}}ds,\quad u\in \mathbb R.
$$
\end{lem}
\begin{proof}
To prove \eqref{psi} we write
\begin{align*}
    \psi (t)&=\int_t^1\frac{1}{(1+t^2/\rho ^2)^{m+1/2}}\frac{d\rho}{\rho }=\int_t^1\frac{1}{(1+u^2)^{m+1/2}}\frac{du}{u}=\int_t^1\Big(\frac{1}{(1+u^2)^{m+1/2}}-1\Big)\frac{du}{u}+\int_t^1\frac{du}{u}\\
    &=\int_t^1\Big(\frac{1}{(1+u^2)^{m+1/2}}-1\Big)\frac{du}{u}-\log t,\quad t\in (0,1).
\end{align*}
By the mean value theorem,
$$
\left|\Big(\frac{1}{(1+u^2)^{m+1/2}}-1\Big)\frac{1}{u}\right|=\frac{|1-(1+u^2)^{m+1/2}|}{u(1+u^2)^{m+1/2}}\leq C\frac{u}{1+u^2}\leq Cu,\quad u\in (0,1),
$$
and the conclusion follows from the dominated convergence theorem.

To prove \eqref{varphi} we first verify that the integral defining $\beta_m$ is convergent. Indeed, the mean value theorem allows us to write  
\begin{align*}
\frac{|H_m(u)-H_m(0)|}{u}&=\frac{1}{u}\left|\int_1^\infty \frac{s}{\sqrt{s^2-1}}\frac{s^{2m+1}-(s^2+u^2)^{m+1/2}}{s^{2m+1}(s^2+u^2)^{m+1/2}}ds\right|\\
&\leq Cu\int_1^\infty  \frac{1}{\sqrt{s^2-1}}\frac{(s^2+u^2)^{m-1/2}}{s^{2m}(s^2+u^2)^{m+1/2}}ds\\
&\leq Cu\int_1^\infty \frac{ds}{\sqrt{s-1}s^{2m+2}}\leq Cu,\quad u\in (0,1).
\end{align*}
Now, adequate changes of variables lead to
\begin{align*}
\varphi (t)&=\int_t^1\int_1 ^\infty \frac{s\rho ^{2m}}{\sqrt{s^2-1}(s^2\rho ^2+t^2)^{m+1/2}}dsd\rho=\int_t^1\int_1 ^\infty \frac{s}{\sqrt{s^2-1}(s^2+u^2)^{m+1/2}}\frac{dsdu}{u}\\
&=\int_t^1H_m(u)\frac{du}{u}=\int_t^1\frac{H_m(u)-H_m(0)}{u}du-H_m(0)\log t,\quad t\in (0,1),
\end{align*}
and \eqref{varphi} follows.
\end{proof}

Finally, as we have seen, the modified Bessel function $K_\nu $, $\nu >0$, arises naturally in the hyperbolic setting.  We collect the following fundamental properties for this function (see, for example,  \cite[Chapter 5]{Leb}).
\begin{equation}\label{Kderiv}
\Big(\frac{1}{z}\frac{\partial }{\partial z}\Big)^r[z^{-\nu }K_\nu (z)]=(-1)^rz^{-(\nu +r)}K_{\nu +r}(z),\quad z>0.
\end{equation}
\begin{equation}\label{cotaKinfty}
\lim_{z\rightarrow \infty }\frac{K_\nu (z)}{\sqrt{\frac{\pi}{2z}}e^{-z}}=1,\quad \nu>0.
\end{equation}
\begin{equation}\label{asintotico}
K_\nu (z)-\sqrt{\frac{\pi}{2z}}e^{-z}=O\Big(\frac{e^{-z}}{z^{3/2}}\Big),\quad \mbox{ as }z\rightarrow \infty.
\end{equation}

\section{Proof of Theorem \ref{Th1.1}}\label{S3}

Let $f\in C_c(\mathbb H^n)\cap {\rm Lip}_{\rm loc}^\alpha (\mathbb H^n)$. For each $s\in (0,\alpha/p')$, where $p'$ denotes the conjugate exponent of $p$, we have that 
$$
(-\Delta_{\mathbb H^n})_p^sf(x)=\int_{\mathbb H^n}\Phi _p(f(x)-f(y))\mathcal K_{n,s,p}(d_{\mathbb H^n}(x,y))dy,\quad x\in \mathbb H^n,
$$
being the integral absolutely convergent for every $x\in \mathbb H^n$. Indeed, fix $x\in \mathbb H^n$ and take $R\geq 1$ such that $\supp f\subseteq B(x,R)$. We consider the following two integrals
\begin{equation}\label{I1}
I_1(s)=\int_{B(x,R)}\Phi _p(f(x)-f(y))\mathcal K_{n,s,p}(d_{\mathbb H^n}(x,y))dy,\quad s\in (0,\alpha/p'), 
\end{equation}
and
\begin{equation}\label{I2}
I_2(s)=\int_{\mathbb H^n\setminus B(x,R)}\Phi _p(f(x)-f(y))\mathcal K_{n,s,p}(d_{\mathbb H^n}(x,y))dy,\quad s\in (0,\alpha/p').
\end{equation}
By taking into account that $f\in {\rm Lip }_{\rm loc}^\alpha (\mathbb H^n)$ we have that
$$
    |I_1(s)|\leq C\|f\|_{{\rm Lip}^\alpha (B(x,R))}^{p-1}\int_{B(x,R)}(d_{\mathbb H^n}(x,y))^{\alpha (p-1)}\mathcal K_{n,s,p}(d(x,y))dy.
$$
Let $g_x\in SO(1,n)$ such that $g_x(e_0)=x$. Since $g_x$ is an isometry we can write
$$
|I_1(s)|\leq C\int_{B(e_0,R)}(d_{\mathbb H^n}(e_0,z))^{\alpha (p-1)}\mathcal K_{n,s,p}(d_{\mathbb H^n}(e_0,z))dz\leq C\int_0^R\rho ^{\alpha (p-1)}\mathcal K_{n,s,p}(\rho )(\sinh \rho )^{n-1}d\rho .
$$
According to \eqref{K0} we obtain
\begin{align}\label{cotaI1}
|I_1(s)|&\leq Cs\int_0^R\rho ^{\alpha (p-1)-n-sp}(\sinh \rho )^{n-1}d\rho \leq Cs\int_0^R\rho ^{(\alpha/p'-s)p-1}d\rho \nonumber\\
&=C\frac{sR^{(\alpha/p'-s)p}}{(\alpha/p'-s)}<\infty,\quad 0<s<\frac{\alpha }{p'}.
\end{align}

On the other hand, since $f$ is bounded in $\mathbb H^n$, using \eqref{Kinfty} we get, for $s\in (0,\alpha/p')$, 
\begin{align*}
|I_2(s)|&\leq C\int_R^\infty \mathcal K_{n,s,p}(\rho )(\sinh \rho )^{n-1}d\rho \leq Cs\int_R^\infty \rho ^{-1-sp/2}e^{-(n-1)\rho }e^{(n-1)\rho }d\rho \\
&=Cs\int_R^\infty \rho ^{-1-sp/2}d\rho =\frac{C}{R^{sp/2}}<\infty.
\end{align*}

Our next objective is to see that there exists $A_{n,p}>0$ such that 
\begin{equation}\label{objlim}
\lim_{s\rightarrow 0^+}(-\Delta_{\mathbb H^n})_p^sf(x)=A_{n,p}\Phi _p(f(x)),\quad x\in\mathbb H^n.
\end{equation}

Let $x\in \mathbb H^n$ and choose, as before, $R\geq 1$ such that ${\rm supp }\,f\subseteq B(x,R)$. We write 
$$
(-\Delta_{\mathbb H^n})_p^sf(x)=I_1(s)+I_2(s), \quad s\in \big(0,\tfrac{\alpha}{p'}\big),
$$
where $I_1$ and $I_2$ are defined by \eqref{I1} and \eqref{I2}, respectively.

Note that, by virtue of \eqref{cotaI1}, $\lim_{s\rightarrow 0^+}I_1(s)=0$. On the other hand, taking into account that $f(y)=0$, $y\in \mathbb H^n\setminus B(x,R)$, we have  
\begin{equation}\label{I2b}
I_2(s)=\Phi _p(f(x))\int_{\mathbb H^n\setminus B(x,R)}\mathcal K_{n,s,p}(d_{\mathbb H^n}(x,y))dy,\quad s\in \big(0,\tfrac{\alpha }{p'}\big). 
\end{equation}
Thus, we have to show that there exists a constant $A_{n,p}>0$, not depending on $f$ or $x$, such that 
\begin{equation}\label{Anp}
    \lim_{s\rightarrow 0^+}\int_{\mathbb H^n\setminus B(x,R)}\mathcal K_{n,s,p}(d_{\mathbb H ^n}(x,y))dy=A_{n,p}.
\end{equation}

Note that
$$
\int_{\mathbb H^n\setminus B(x,R)}\mathcal K_{n,s,p}(d_{\mathbb H^n}(x,y))dy=c_n\int_R^\infty \mathcal K_{n,s,p}(\rho)(\sinh \rho )^{n-1}d\rho,
$$
where $c_n=2\pi^{n/2}/\Gamma (n/2)$.

\subsection{Case $n=2m+1$, $m\in \mathbb N$} Using Lemma \ref{key} we can write
\begin{align}\label{S1S2}
    \int_R^\infty \mathcal K_{n,s,p}(\rho )(\sinh \rho )^{n-1}d\rho&\nonumber\\
    &\hspace{-3cm}=(-1)^msa_{n,s,p}\sum_{k=1}^m\sum_{j=0}^{k-1}\int_R^\infty \frac{Q_j^{m,k}(\rho )\rho ^j}{(\sinh \rho )^{2m-k}}\Big(\frac{1}{\rho }\frac{\partial }{\partial \rho }\Big)^k\big[\rho ^{-\frac{1+sp}{2}}K_{\frac{1+sp}{2}}(m\rho )\big](\sinh \rho )^{2m}d\rho \nonumber\\
    &\hspace{-3cm}\quad +(-1)^msa_{n,s,p}\sum_{k=1}^m\int_R^\infty \frac{Q_k^{m,k}(\rho )\rho ^k}{(\sinh \rho )^{2m-k}}\Big(\frac{1}{\rho }\frac{\partial }{\partial \rho }\Big)^k\big[\rho ^{-\frac{1+sp}{2}}K_{\frac{1+sp}{2}}(m\rho )\big](\sinh \rho )^{2m}d\rho \nonumber\\
    &\hspace{-3cm}:=(-1)^msa_{n,s,p}(S_1(s)+S_2(s)),\quad s>0.
\end{align}
Let $k=1,\ldots, m$. For every $j=0,\ldots, k-1$, according to Lemma \ref{key} $(iii)$, \eqref{Kderiv}, and \eqref{cotaKinfty} we have
\begin{align}\label{sS1}
    \Big|\int_R^\infty \frac{Q_j^{m,k}(\rho )\rho ^j}{(\sinh \rho )^{2m-k}}\Big(\frac{1}{\rho }\frac{\partial }{\partial \rho }\Big)^k\big[\rho ^{-\frac{1+sp}{2}}K_{\frac{1+sp}{2}}(m\rho )\big](\sinh \rho )^{2m}d\rho\Big|&\nonumber\\
    &\hspace{-7cm}\leq C\int_R^\infty e^{m\rho }\rho ^{j-\frac{1+sp}{2}-k}|K_{\frac{1+sp}{2}+k}(m\rho)|d\rho\leq C\int_R^\infty \rho ^{j-k-1}d\rho =C\frac{R^{j-k}}{k-j}\leq C,
\end{align}
where $C$ does not depend on $s$. Then, 
\begin{equation}\label{S1}
\lim_{s\rightarrow 0^+}(-1)^msa_{n,s,p}S_1(s)=0.
\end{equation}
On the other hand, we use \eqref{Kderiv} and make the following decomposition:
\begin{align}\label{Iik}
    \int_R^\infty \frac{Q_k^{m,k}(\rho )\rho ^k}{(\sinh \rho )^{2m-k}}\Big(\frac{1}{\rho }\frac{\partial }{\partial \rho }\Big)^k\big[\rho ^{-\frac{1+sp}{2}}K_{\frac{1+sp}{2}}(m\rho )\big](\sinh \rho )^{2m}d\rho\nonumber\\
    &\hspace{-7cm}=(-1)^km^k\int_R^\infty Q_k^{m,k}(\rho )(\sinh \rho )^k\rho ^{-\frac{1+sp}{2}}K_{\frac{1+sp}{2}+k}(m\rho )d\rho\nonumber\\
    &\hspace{-7cm}=(-1)^km^k\int_R^\infty [Q_k^{m,k}(\rho )-c_k^{m,k}e^{-m\rho }(\sinh \rho )^{2m-k}](\sinh \rho )^k\rho ^{-\frac{1+sp}{2}}K_{\frac{1+sp}{2}+k}(m\rho )d\rho\nonumber\\
    &\hspace{-7cm}\quad +(-1)^km^kc_k^{m,k}\int_R^\infty e ^{-m\rho }(\sinh \rho )^{2m}\rho ^{-\frac{1+sp}{2}}\big[K_{\frac{1+sp}{2}+k}(m\rho )-\frac{\sqrt{\pi}}{\sqrt{2m\rho}}e^{-m\rho}\big]d\rho \nonumber\\
    &\hspace{-7cm}\quad +(-1)^km^kc_k^{m,k}\frac{\sqrt{\pi}}{\sqrt{2m}}\int_R^\infty e^{-2m\rho }\rho ^{-1-\frac{sp}{2}}\big[(\sinh \rho )^{2m}-\frac{e^{2m\rho}}{2^{2m}}\big]d\rho\nonumber\\
    &\hspace{-7cm}\quad +(-1)^km^kc_k^{m,k}\frac{\sqrt{\pi}}{2^{2m}\sqrt{2m}}\int_R^\infty \rho ^{-1-\frac{sp}{2}}d\rho\nonumber\\
    &\hspace{-7cm}=:\sum_{i=1}^4 I_i^{(k)}(s),\quad s>0.
\end{align}
It is clear that 
\begin{equation}\label{I4k}
I_4^{(k)}(s)=(-1)^km^kc_k^{m,k}\frac{\sqrt{2\pi}}{2^{2m}\sqrt{m}}\frac{R^{-\frac{sp}{2}}}{sp},\quad s>0.
\end{equation}
We claim that, when $i=1,2,3$,
\begin{equation}\label{I123}
\lim_{s\rightarrow 0^+}sI_i^{(k)}(s)=0.
\end{equation}
Hence, using also that $(-1)^{m+k}c_k^{m,k}=|c_k^{m,k}|$ (Lemma \ref{key} $(ii)$),
\begin{align}\label{nodd}
\lim_{s\rightarrow 0^+}(-1)^msa_{n,s,p}S_2(s)&=(-1)^mc_{n,p}\sum_{k=1}^m\lim_{s\rightarrow 0^+}sI_4^{(k)}(s)=\frac{\sqrt{2\pi}c_{n,p}}{2^{2m}\sqrt{m}p}\sum_{k=1}^m(-1)^{m+k}c_k^{m,k}m^k\nonumber\\
&=\frac{\sqrt{2\pi}c_{n,p}}{2^{2m}\sqrt{m}p}\sum_{k=1}^m|c_k^{m,k}|m^k,
\end{align}
which, jointly with \eqref{S1S2} and \eqref{S1}, leads to \eqref{Anp} with
\begin{equation}\label{Anpodd}
A_{n,p}=\frac{c_n\sqrt{2\pi}c_{n,p}}{2^{2m}\sqrt{m}p}\sum_{k=1}^m|c_k^{m,k}|m^k>0,\quad n=2m+1.
\end{equation}
Let us prove \eqref{I123}. According to Lemma \ref{key} $(ii)$ and \eqref{cotaKinfty} we obtain that
\begin{equation}\label{sI1k}
    |sI_1^{(k)}(s)|\leq Cs\int_R^\infty (\sinh \rho )^{2m}e^{-(m+2)\rho }\rho ^{-1-\frac{sp}{2}}e^{-m\rho}d\rho\leq \frac{Cs}{R}\int_R^\infty e^{-2\rho}d\rho\leq Cs,\quad  s\in (0,1).
\end{equation}
Using \eqref{asintotico} we get
\begin{equation}\label{sI2k}
|sI_2^{(k)}(s)|\leq Cs\int_R^\infty e^{-2m\rho }(\sinh \rho )^{2m}\rho ^{-2-\frac{sp}{2}}d\rho \leq Cs\int_R^\infty \rho ^{-2}d\rho =\frac{Cs}{R}\leq Cs, \quad s\in (0,1).
\end{equation}
Finally, since $(\sinh \rho )^{2m}=2^{-2m}e^{2m\rho}(1-e^{-2\rho })^{2m}$, $\rho >0$, it follows that
\begin{align}\label{sI3k}
|sI_3^{(k)}(s)|&\leq Cs\int_R^\infty \rho ^{-1-\frac{sp}{2}}\big(1-(1-e^{-2\rho })^{2m}\big)d\rho \leq Cs\int_R^\infty \rho ^{-1}\sum_{j=1}^{2m}(-1)^{j+1}\binom{2m}{j}e^{-2j\rho }d\rho\nonumber\\
&\leq \frac{Cs}{R}\int_R^\infty e^{-2\rho }d\rho\leq Cs,\quad s\in (0,1).
\end{align}
Note that the constants $C$ appearing in the above estimations can be choosen independent of $s$ and $R$. Hence, \eqref{I123} is established.

\subsection{Case $n=2m$, $m\in \mathbb N$}\label{Seven} Again, taking into account Lemma \ref{key} we write
\begin{align}\label{S1S2even}
    \int_R^\infty \mathcal K_{n,s,p}(\rho )(\sinh \rho )^{n-1}d\rho\\
    &\hspace{-3.5cm}=\frac{(-1)^msa_{n,s,p}}{\sqrt{\pi}}\int_R^\infty\int_\rho ^\infty \frac{(\sinh \rho )^{2m-1}\sinh r}{\sqrt{\cosh r-\cosh \rho}}\Big(\frac{1}{\sinh r }\frac{\partial }{\partial r}\Big)^m\big[r ^{-\frac{1+sp}{2}}K_{\frac{1+sp}{2}}\big(\tfrac{2m-1}{2}r\big)\big]drd\rho \nonumber\\
    &\hspace{-3.5cm}=\frac{(-1)^msa_{n,s,p}}{\sqrt{2\pi}}\int_R^\infty\int_\rho ^\infty \frac{(\sinh \rho )^{2m-1}\sinh r}{\sqrt{\sinh (\frac{r+\rho}{2})\sinh (\frac{r-\rho }{2})}}\Big(\frac{1}{\sinh r }\frac{\partial }{\partial r}\Big)^m\big[r ^{-\frac{1+sp}{2}}K_{\frac{1+sp}{2}}\big(\tfrac{2m-1}{2}r\big)\big]drd\rho \nonumber\\
    &\hspace{-3.5cm}=\frac{(-1)^msa_{n,s,p}}{\sqrt{2\pi}} \big(T_1(s)+T_2(s)\big),\quad s\in (0,1),\nonumber
    \end{align}
where, for every $s>0$, 
$$
T_1(s)=\sum_{k=1}^m\sum_{j=0}^{k-1}\int_R^\infty \int_\rho ^\infty  \frac{(\sinh \rho )^{2m-1}\sinh r}{\sqrt{\sinh (\frac{r+\rho}{2})\sinh (\frac{r-\rho }{2})}}\frac{Q_j^{m,k}(r)r^j}{(\sinh r)^{2m-k}}\Big(\frac{1}{r}\frac{\partial }{\partial r}\Big)^k\big[r ^{-\frac{1+sp}{2}}K_{\frac{1+sp}{2}}\big(\tfrac{2m-1}{2}r\big)\big]drd\rho,
$$
and
$$
T_2(s)=\sum_{k=1}^m\int_R^\infty \int_\rho ^\infty \frac{(\sinh \rho )^{2m-1}\sinh r}{\sqrt{\sinh (\frac{r+\rho}{2})\sinh (\frac{r-\rho }{2})}}\frac{Q_k^{m,k}(r)r^k}{(\sinh r)^{2m-k}}\Big(\frac{1}{r}\frac{\partial }{\partial r}\Big)^k\big[r ^{-\frac{1+sp}{2}}K_{\frac{1+sp}{2}}\big(\tfrac{2m-1}{2}r\big)\big]drd\rho.
$$
Let us establish that $\lim_{s\rightarrow 0^+}sa_{n,s,p}T_1(s)=0$. Consider $k=1,\ldots, m$, and $j=0,\ldots, k-1$. First observe that, for each $a>0$, there exists $C>0$ such that  
\begin{equation}\label{sinsin}
\sinh \big(\tfrac{r+\rho}{2}\big)\sinh \big(\tfrac{r-\rho }{2}\big)=\frac{e^r}{4}(1-e^{-(r+\rho)})(1-e^{-(r-\rho )})\geq Ce^r(1-e^{-(r-\rho )}), \quad r+\rho >a.
\end{equation}
Then, by virtue of Lemma \ref{key} $(iii)$, \eqref{Kderiv} and \eqref{cotaKinfty}, we have that, there exists $C>0$ such that
\begin{align*}
    \Big|\int_R^\infty \int_\rho ^\infty\frac{(\sinh \rho )^{2m-1}\sinh r}{\sqrt{\sinh (\frac{r+\rho}{2})\sinh (\frac{r-\rho }{2})}}\frac{Q_j^{m,k}(r)r^j}{(\sinh r)^{2m-k}}\Big(\frac{1}{r}\frac{\partial }{\partial r}\Big)^k\big[r ^{-\frac{1+sp}{2}}K_{\frac{1+sp}{2}}\big(\tfrac{2m-1}{2}r\big)\big]drd\rho \Big|&\\
    &\hspace{-12cm}\leq C\int_R^\infty \int_\rho ^\infty \frac{e^{(2m-1)\rho }e^{(\frac{1}{2}-m)r}}{\sqrt{1-e^{-(r-\rho )}}}r^{j-k-\frac{1+sp}{2}}\big|K_{\frac{1+sp}{2}+k}
    \big(\tfrac{2m-1}{2}r\big)\big|drd\rho \\
    &\hspace{-12cm}\leq C\int_R^\infty \int_\rho ^\infty \frac{e^{(2m-1)\rho }e^{(\frac{1}{2}-m)r}}{\sqrt{1-e^{-(r-\rho )}}}r^{j-k-1-\frac{sp}{2}}e^{-(m-\frac{1}{2})r}drd\rho\\
    &\hspace{-12cm}\leq C\int_R^\infty \rho ^{j-k-1-\frac{sp}{2}}e^{(2m-1)\rho }\int_\rho ^\infty \frac{e^{-(2m-1)r}}{\sqrt{1-e^{-(r-\rho )}}}drd\rho\\
    &\hspace{-12cm}=C\int_R^\infty \rho ^{j-k-1-\frac{sp}{2}}\int_0 ^\infty \frac{e^{-(2m-1)t}}{\sqrt{1-e^{-t}}}dtd\rho\leq C\int_R^\infty \rho ^{j-k-1-\frac{sp}{2}}d\rho\int_0 ^\infty \frac{e^{-t}}{\sqrt{1-e^{-t}}}dt \\
    &\hspace{-12cm}=C\frac{R^{j-k-\frac{sp}{2}}}{k-j+\frac{sp}{2}} \leq C,\quad s\in (0,1).
\end{align*}
Thus, we can deduce that 
\begin{equation}\label{sT1}
\lim_{s\rightarrow 0^+}sa_{n,s,p}T_1(s)=0.
\end{equation}

On the other hand, let $k=1,\ldots, m$. We can decompose the integral in $T_2(s)$ as follows:
\begin{align}\label{Jik}
    &\int_R^\infty \int_\rho ^\infty \frac{(\sinh \rho )^{2m-1}\sinh r}{\sqrt{\sinh (\frac{r+\rho}{2})\sinh (\frac{r-\rho }{2})}}\frac{Q_k^{m,k}(r)r^k}{(\sinh r)^{2m-k}}\Big(\frac{1}{r}\frac{\partial }{\partial r}\Big)^k\big[r ^{-\frac{1+sp}{2}}K_{\frac{1+sp}{2}}\big(\tfrac{2m-1}{2}r\big)\big]drd\rho\nonumber\\
    &=(-1)^k\big(\tfrac{2m-1}{2}\big)^k\int_R^\infty \int_\rho ^\infty \frac{(\sinh \rho )^{2m-1}\sinh r}{\sqrt{\sinh (\frac{r+\rho}{2})\sinh (\frac{r-\rho }{2})}}\frac{Q_k^{m,k}(r)}{(\sinh r)^{2m-k}}r^{-\frac{1+sp}{2}}K_{\frac{1+sp}{2}+k}\big(\tfrac{2m-1}{2}r\big)drd\rho\nonumber\\
    &=(-1)^k\big(\tfrac{2m-1}{2}\big)^k\sum_{i=1}^5J_i^{(k)}(s),\quad s\in (0,1),
\end{align}
where, for every $s\in (0,1)$,
\begin{align*}
&J_1^{(k)}(s)=\int_R^\infty \!\!\!\int_\rho ^\infty \frac{(\sinh \rho )^{2m-1}\sinh r}{\sqrt{\sinh (\frac{r+\rho}{2})\sinh (\frac{r-\rho }{2})}}\Big[\frac{Q_k^{m,k}(r)}{(\sinh r)^{2m-k}}-c_k^{m,k}e^{-mr}\Big]r ^{-\frac{1+sp}{2}}K_{\frac{1+sp}{2}+k}\big(\tfrac{2m-1}{2}r\big)drd\rho,\\
&J_2^{(k)}(s)=c_k^{m,k}\int_R^\infty \!\!\!\int_\rho ^\infty\frac{(\sinh \rho )^{2m-1}\sinh r}{\sqrt{\sinh (\frac{r+\rho}{2})\sinh (\frac{r-\rho }{2})}}e^{-mr}r^{-\frac{1+sp}{2}}\Big[K_{\frac{1+sp}{2}+k}\big(\tfrac{2m-1}{2}r\big)-\frac{\sqrt{\pi}e^{-\frac{2m-1}{2}r}}{\sqrt{(2m-1)r}}\Big]drd\rho,\\
&J_3^{(k)}(s)=\frac{c_k^{m,k}\sqrt{\pi}}{\sqrt{2m-1}} \int_R^\infty \!\!\!\int_\rho ^\infty\frac{(\sinh \rho )^{2m-1}}{\sqrt{\sinh (\frac{r+\rho}{2})\sinh (\frac{r-\rho }{2})}}e^{-(2m-\frac{1}{2})r}r^{-1-\frac{sp}{2}}\big[\sinh r-\frac{e^r}{2}\big]drd\rho,\\
&J_4^{(k)}(s)=\frac{c_k^{m,k}\sqrt{\pi}}{2\sqrt{2m-1}}\int_R^\infty \!\!\!\int_\rho ^\infty\frac{e^{-(2m-\frac{3}{2})r}r^{-1-\frac{sp}{2}}}{\sqrt{\sinh (\frac{r+\rho}{2})\sinh (\frac{r-\rho }{2})}}\Big[(\sinh \rho )^{2m-1}-\Big(\frac{e^\rho}{2}\Big)^{2m-1}\Big]drd\rho,
\end{align*}
and
$$
J_5^{(k)}(s)=\frac{c_k^{m,k}\sqrt{\pi}}{2^{2m}\sqrt{2m-1}}\int_R^\infty \!\!\!\int_\rho ^\infty\frac{e^{(2m-1)\rho }e^{-(2m-\frac{3}{2})r}r^{-1-\frac{sp}{2}}}{\sqrt{\sinh (\frac{r+\rho}{2})\sinh (\frac{r-\rho }{2})}}drd\rho.
$$
Let us show that $\lim_{s\rightarrow 0^+}sJ_i^{(k)}(s)=0$, when $i=1,\ldots, 4$. 

Taking into account again \eqref{sinsin} and using Lemma \ref{key} $(ii)$ and \eqref{cotaKinfty} it follows that
\begin{align*}
|J_1^{(k)}(s)|&\leq C\int_R^\infty \int_\rho ^\infty \frac{e^{(2m-1)\rho }e^{-(2m+1)r}}{\sqrt{1-e^{-(r-\rho )}}}r^{-1-\frac{sp}{2}}drd\rho\\
&\leq \frac{C}{R}\int_R^\infty e^{-2\rho }d\rho\int_0 ^\infty \frac{e^{-t}}{\sqrt{1-e^{-t}}}dt\leq C,\quad s\in (0,1).
\end{align*}
Also, according to \eqref{asintotico} we obtain
\begin{align*}
|J_2^{(k)}(s)|&\leq C\int_R^\infty \int_\rho ^\infty \frac{e^{(2m-1)\rho }e^{-(2m-1)r}}{\sqrt{1-e^{-(r-\rho )}}}r^{-2-\frac{sp}{2}}drd\rho\\
&\leq C\int_R^\infty \rho^{-2}d\rho\int_0 ^\infty \frac{e^{-t}}{\sqrt{1-e^{-t}}}dt=\frac{C}{R}\leq C,\quad s\in (0,1).
\end{align*}
By taking into account that, for each $\ell \in \mathbb N$,
$$
\Big|(\sinh z)^\ell -\Big(\frac{e^z}{2}\Big)^\ell \Big|=\Big(\frac{e^z}{2}\Big)^\ell \Big|(1-e^{-2z})^\ell -1\Big|=\Big(\frac{e^z}{2}\Big)^\ell\Big|\sum_{i=1}^\ell (-1)^{i}\binom{\ell }{i}e^{-2iz}\Big|\leq Ce^{(\ell -2)z},\quad z>0,
$$
it follows that
\begin{align*}
|J_3^{(k)}(s)|+|J_4^{(k)}(s)|&\leq C\int_R^\infty \int_\rho ^\infty \frac{e^{(2m-1)\rho }e^{-(2m+1)r}+e^{(2m-3)\rho }e^{-(2m-1)r}}{\sqrt{1-e^{-(r-\rho )}}}r^{-1-\frac{sp}{2}}drd\rho \\
&\leq \frac{C}{R}\int_R^\infty e^{-2\rho }d\rho \int_0 ^\infty \frac{e^{-t}}{\sqrt{1-e^{-t}}}dt\leq C,\quad s\in(0,1).
\end{align*}
The constants $C>0$ in the above estimations does not depend on $s$. Hence, $sJ_i^{(k)}(s)\longrightarrow 0$, as $s\rightarrow 0^+$, for $i=1,\ldots ,4$.

Let us now consider $J_5^{(k)}$. Using the first equality in \eqref{sinsin} and taking into account that 
$$
\int_0^\infty \frac{e^{-(2m-1)t}}{\sqrt{1-e^{-t}}}dt=\frac{\sqrt{\pi}\Gamma (2m-1)}{\Gamma (2m-\frac{1}{2})},
$$
we write
\begin{align}\label{Jik5}
    J_5^{(k)}(s)&=\frac{c_k^{m,k}\sqrt{\pi}}{2^{2m-1}\sqrt{2m-1}}\int_R^\infty\int_\rho ^\infty\frac{e^{(2m-1)\rho }e^{-(2m-1)r}r^{-1-\frac{sp}{2}}}{\sqrt{(1-e^{-(r+\rho)})(1-e^{-(r-\rho )})}}drd\rho\nonumber\\
    &=\frac{c_k^{m,k}\sqrt{\pi}}{2^{2m-1}\sqrt{2m-1}}\int_R^\infty\int_0^\infty\frac{e^{-(2m-1)t}(t+\rho)^{-1-\frac{sp}{2}}}{\sqrt{(1-e^{-(t+2\rho)})(1-e^{-t})}}dtd\rho\nonumber\\
    &=\frac{c_k^{m,k}\sqrt{\pi}}{2^{2m-1}\sqrt{2m-1}}\Big(\int_R^\infty\int_0^\infty\frac{e^{-(2m-1)t}(t+\rho)^{-1-\frac{sp}{2}}}{\sqrt{1-e^{-t}}}\Big(\frac{1}{\sqrt{1-e^{-(t+2\rho})}}-1\Big)dtd\rho\nonumber\\
    &\quad +\int_R^\infty\int_0^\infty\frac{e^{-(2m-1)t}}{\sqrt{1-e^{-t}}}\Big((t+\rho)^{-1-\frac{sp}{2}}-\rho ^{-1-\frac{sp}{2}}\Big)dtd\rho+\frac{\sqrt{\pi}\Gamma (2m-1)}{\Gamma (2m-\frac{1}{2})}\int_R^\infty\rho ^{-1-\frac{sp}{2}}d\rho \Big)\nonumber\\
    &=\frac{c_k^{m,k}\sqrt{\pi}}{2^{2m-1}\sqrt{2m-1}}\sum_{i=1}^3 J_i^{k,5}(s),\quad s\in (0,1).
\end{align}
Since
$$
0\leq \frac{1}{\sqrt{1-e^{-z}}}-1=\frac{e^{-z}}{\sqrt{1-e^{-z}}(1+\sqrt{1-e^{-z}})}\leq Ce^{-z},\quad z>1,
$$
we obtain
$$
|J_1^{k,5}(s)|\leq \frac{C}{R}\int_R^\infty e^{-2\rho }d\rho \int_0^\infty \frac{e^{-2mt}}{\sqrt{1-e^{-t}}}dt\leq C,\quad s\in (0,1).
$$
On the other hand, by using mean value theorem we get that there exists $C>0$ such that
$$
\big|(t+\rho )^{-1-\frac{sp}{2}}-\rho ^{-1-\frac{sp}{2}}\big|\leq Ct\rho ^{-2-\frac{sp}{2}}\leq Ct\rho ^{-2},\quad t>0,\;\rho >1,\mbox{ and }s\in (0,1).
$$
We deduce that
$$
|J_2^{k,5}(s)|\leq C\int_R^\infty \rho ^{-2}d\rho \int_0^\infty \frac{te^{-(2m-1)t}}{\sqrt{1-e^{-t}}}dt\leq \frac{C}{R}\leq C,\quad s\in (0,1).
$$
Finally, observe that
\begin{equation}\label{J3k5}
J_3^{k,5}(s)=\frac{2\sqrt{\pi}}{sp}\frac{\Gamma(2m-1)}{\Gamma (2m-\frac{1}{2})}R^{-\frac{sp}{2}},\quad s\in (0,1).
\end{equation}
With these estimations we can infer
$$
\lim_{s\rightarrow 0^+}sa_{n,s,p}J_5^{(k)}(s)=\frac{c_{n,p}c_k^{m,k}\pi }{2^{2m-2}p\sqrt{2m-1}}\frac{\Gamma (2m-1)}{\Gamma (2m-\frac{1}{2})}.
$$
Hence,
\begin{align}\label{sT2}
\lim_{s\rightarrow 0^+}\frac{(-1)^msa_{n,s,p}}{\sqrt{2\pi}}T_2(s)&=\frac{1}{\sqrt{2\pi}}\sum_{k=1}^m(-1)^{m+k}\big(\tfrac{2m-1}{2}\big)^k\lim_{s\rightarrow 0^+}sa_{n,s,p}J^{(k)}_5(s)\nonumber\\
&=\frac{\sqrt{\pi }c_{n,p}}{2^{2m-3/2}p\sqrt{2m-1}}\frac{\Gamma (2m-1)}{\Gamma (2m-\frac{1}{2})}\sum_{k=1}^m(-1)^{m+k}c_k^{m,k}\big(\tfrac{2m-1}{2}\big)^k.
\end{align}
From \eqref{sT1} and \eqref{sT2} we obtain that, when $n$ is even, \eqref{Anp} holds with
\begin{equation}\label{Aneven}
A_{n,p}=\frac{c_n\sqrt{\pi }c_{n,p}}{2^{n-3/2}p\sqrt{n-1}}\frac{\Gamma (n-1)}{\Gamma (n-\frac{1}{2})}\sum_{k=1}^m|c_k^{m,k}|\big(\tfrac{2m-1}{2}\big)^k,\quad n=2m.
\end{equation}

\section{Proof of Theorem \ref{Th1.2}}\label{S4}

{\it (a)} Let $x\in \mathbb H^n$ and assume that $\supp f\subseteq B(x,R)$, with $R\geq 1$. As in the proof of Theorem \ref{Th1.1} we can write
$(-\Delta_{\mathbb H^n})_p^sf(x)=I_1(s)+I_2(s)$, $s\in (0,\alpha /p')$, where $I_1$ and $I_2$ are the absolutely convergent integrals given by \eqref{I1} and \eqref{I2b}, respectively. Then,
\begin{equation}\label{cocloga}
\frac{(-\Delta_{\mathbb H^n})_p^sf(x)-A_{n,p}\Phi _p(f(x))}{s}=\frac{I_1(s)}{s}+\frac{I_2(s)-A_{n,p}\Phi _p(f(x))}{s},\quad s\in \big(0,\tfrac{\alpha}{p'}\big).
\end{equation}
By using \eqref{K0}, and that $f\in {\rm Lip}_{\rm loc}^\alpha (\mathbb H^n)$ it follows that, when $s\in (0,\tfrac{\alpha}{2p'})$, 
\begin{align*}
\frac{1}{s}|\Phi _p(f(x)-f(y))\mathcal K_{n,s,p}(d_{\mathbb H^n}(x,y))|&\leq Cd_{\mathbb H^n}(x,y)^{\alpha(p-1)-n-sp}\\
&\leq C d_{\mathbb H^n}(x,y)^{\frac{\alpha(p-1)}{2}-n},\quad y\in B(x,R).
\end{align*}
Moreover,
$$
\int_{B(x,R)}d_{\mathbb H^n}(x,y)^{\frac{\alpha(p-1)}{2}-n}dy\leq C\int_0^R\rho ^{\frac{\alpha(p-1)}{2}-n}(\sinh \rho )^{n-1}d\rho \leq C\int_0^R\rho ^{\frac{\alpha(p-1)}{2}-1}d\rho <\infty.
$$
Hence, we can apply the dominated convergence theorem which, jointly \eqref{limKP}, allows us to get
\begin{equation}\label{limI1s}
\lim_{s\rightarrow 0^+}\frac{I_1(s)}{s}=\sigma_{n,p}\int_{B(x,R)}\Phi _p(f(x)-f(y))P_n(d_{\mathbb H^n}(x,y) ,0)dy.
\end{equation}
Next, we analyze the term,
$$
D(s)=\frac{I_2(s)-A_{n,p}\Phi _p(f(x))}{s},\quad s\in \big(0,\tfrac{\alpha}{p'}\big).
$$
To deal with $I_2$ we will consider the decomposition given in the proof of Theorem \ref{Th1.1}. 

\subsection{Case $n=2m+1$, $m\in \mathbb N$}\label{Dodd} We write (see \eqref{S1S2} and \eqref{Iik})
\begin{align}\label{D}
    D(s)&=\frac{\Phi _p(f(x))}{s}\left(\int_{\mathbb H^n\setminus B(x,R)}\mathcal K_{n,s,p}(d_{\mathbb H^n}(x,y))dy-A_{n,p}\right)\\
    &=\Phi _p(f(x))c_n(-1)^ma_{n,s,p}\Big(S_1(s)+S_2(s)-\sum_{k=1}^mI_4^{(k)}(s)\Big)\nonumber\\
    &\quad +\Phi_p(f(x))\frac{c_n(-1)^msa_{n,s,p}\sum_{k=1}^mI_4^{(k)}(s)-A_{n,p}}{s}=:D_1(s)+D_2(s),\quad s\in \big(0,\tfrac{\alpha}{p'}\big).\nonumber
\end{align}
As in the proof of \eqref{sS1}, according to Lemma \ref{key} $(iii)$, \eqref{Kderiv}, and \eqref{cotaKinfty}, we have, for every $k=1,\ldots, m$, and $j=0,\ldots, k-1$, 
\begin{align*}
\Big| \frac{Q_j^{m,k}(\rho )\rho ^j}{(\sinh \rho )^{2m-k}}\Big(\frac{1}{\rho }\frac{\partial }{\partial \rho }\Big)^k\big[\rho ^{-\frac{1+sp}{2}}K_{\frac{1+sp}{2}}(m\rho )\big](\sinh \rho )^{2m}\Big|&\leq C\rho ^{j-k-1},\quad \rho \geq 1,\;s\in (0,1).
\end{align*}
Moreover, considering the manipulations made in \eqref{Iik}, and in estimations \eqref{sI1k}, \eqref{sI2k} and \eqref{sI3k} we get, for every $k=1,\ldots, m$, 
\begin{align*}
    \Big|\frac{Q_k^{m,k}(\rho )\rho ^k}{(\sinh \rho )^{2m-k}}\Big(\frac{1}{\rho }\frac{\partial }{\partial \rho }\Big)^k\big[\rho ^{-\frac{1+sp}{2}}K_{\frac{1+sp}{2}}(m\rho )\big](\sinh \rho )^{2m}-(-1)^km^kc_k^{m,k}\frac{\sqrt{\pi}}{2^{2m}\sqrt{2m}}\rho ^{-1-\frac{sp}{2}}\Big|&\\
    &\hspace{-10cm}\leq C(e^{-2\rho}+\rho ^{-2}),\quad \rho \geq 1,\quad s\in (0,1).
\end{align*}
Hence, we can apply the dominated convergence theorem to obtain that, for certain $\alpha_{n,p}^{(1)}\in \mathbb R$,
$$
\lim_{s\rightarrow 0^+}D_1(s)=\alpha_{n,p}^{(1)}\Phi _p(f(x)).
$$
On the other hand, from \eqref{I4k} and \eqref{Anpodd} we have that, for certain $\alpha_{n,p}^{(2)}$ and $\beta_{n,p}\in \mathbb R$,
\begin{align*}
\lim_{s\rightarrow 0^+}D_2(s)&=\Phi_p(f(x))\frac{c_n\sqrt{2\pi}}{2^{2m}\sqrt{m}p}\sum_{k=1}^m|c_k^{m,k}|m^k\lim_{s\rightarrow 0^+}\frac{a_{n,s,p}R^{-\frac{sp}{2}}-c_{n,p}}{s}\nonumber\\
&=\Phi_p(f(x))(\alpha_{n,p}^{(2)}+\beta_{n,p}\log R).
\end{align*}
Hence, there exist $\alpha_{n,p},\beta_{n,p}\in \mathbb R$ such that
$$
\lim_{s\rightarrow 0^+}D(s)=(\alpha_{n,p}+\beta_{n,p}\log R)\Phi_p(f(x)),
$$
that, jointly \eqref{cocloga} and \eqref{limI1s}, allows us to finish the proof of {\it (a)} when $n\in \mathbb N$ is odd, and $n\geq 3$.   

\subsection{Case $n=2m$, $m\in \mathbb N$}\label{Deven} We now can write (see \eqref{S1S2even}, \eqref{Jik}, and \eqref{Jik5})
\begin{align*}
    D(s)&=\Phi _p(f(x))\frac{c_n(-1)^ma_{n,s,p}}{\sqrt{2\pi}}\Big[T_1(s)+T_2(s)-\sum_{k=1}^m(-1)^k\big(\tfrac{2m-1}{2}\big)^k\frac{c_k^{m,k}\sqrt{\pi }}{2^{2m-1}\sqrt{2m-1}}J_3^{k,5}(s)\Big]\\
    &\quad +\frac{\Phi_p(f(x))}{s}\Big[\frac{c_nsa_{n,s,p}}{2^{2m-1/2}\sqrt{2m-1}}\sum_{k=1}^m|c_k^{m,k}|\big(\tfrac{2m-1}{2}\big)^kJ_3^{k,5}(s)-A_{n,p}\Big]\\
    &=:\mathfrak D_1(s)+\mathfrak D_2(s),\quad s\in \big(0,\tfrac{\alpha}{p'}\big).
\end{align*}
By taking into account the manipulations followed to get the estimations for $T_1$ in the proof of Theorem \ref{Th1.1} (see section \ref{Seven}) we get, for every $k=1,\ldots, m$, and $j=0,\ldots, k-1$, 
\begin{align*}
\left|\frac{(\sinh \rho )^{2m-1}\sinh r}{\sqrt{\sinh (\frac{r+\rho}{2})\sinh (\frac{r-\rho }{2})}}\frac{Q_j^{m,k}(r)r^j}{(\sinh r)^{2m-k}}\Big(\frac{1}{r}\frac{\partial }{\partial r}\Big)^k\big[r ^{-\frac{1+sp}{2}}K_{\frac{1+sp}{2}}\big(\tfrac{2m-1}{2}r\big)\big]\right|&\\
&\hspace{-7cm}\leq C \rho^{j-k-1}\frac{e^{(2m-1)\rho }e^{-(2m-1)r}}{\sqrt{1-e^{-(r-\rho )}}}\leq \frac{C}{\rho ^2}\frac{e^{-(r-\rho )}}{\sqrt{1+e^{-(r-\rho )}}},\quad 1\leq \rho \leq r,\;s\in (0,1),
\end{align*}
and, since 
$$
\int_R^\infty\int_\rho^\infty \frac{1}{\rho ^2}\frac{e^{-(r-\rho)}}{\sqrt{1-e^{-(r-\rho )}}}drd\rho=\int_R^\infty \frac{1}{\rho ^2}d\rho\int_0^\infty \frac{e^{-t}}{\sqrt{1-e^{-t}}}dt<\infty,
$$
the dominated convergence theorem leads to, 
$$
\lim_{s\rightarrow 0^+}a_{n,s,p}T_1(s)=\alpha_{n,p}^{(1)},
$$
for certain $\alpha_{n,p}^{(1)}\in \mathbb R$. On the other hand, taking into account the arguments in section \ref{Seven} concerning the estimations for $J_i^{(k)}$, $i=1,\ldots, 4$, and for $J_i^{k,5}$, $i=1,2$, we get
\begin{align*}
T_2(s)-\sum_{k=1}^m (-1)^k\big(\tfrac{2m-1}{2}\big)^k\frac{c_k^{m,k}\sqrt{\pi }}{2^{2m-1}\sqrt{2m-1}}J_3^{k,5}(s)\\
&\hspace{-6cm}=\sum_{k=1}^m (-1)^k\big(\tfrac{2m-1}{2}\big)^k\Big[\sum_{i=1}^4 J_i^{(k)}(s)+\frac{c_k^{m,k}\sqrt{\pi }}{2^{2m-1}\sqrt{2m-1}}\big(J_1^{k,5}(s)+J_2^{k,5}(s)\big)\Big]\\
&\hspace{-6cm}=\int_R^\infty \int_\rho ^\infty H(r,\rho )drd\rho,
\end{align*}
for certain $H(r,\rho )$, $1\leq \rho \leq r$, which satisfies
\begin{align*}
|H(r,\rho)|&\leq C\frac{e^{(2m-1)\rho }e^{-(2m-1)r}}{\sqrt{1-e^{-(r-\rho )}}}\Big(\frac{e^{-2r}}{r}+\frac{1}{r^2}+\frac{e^{-2\rho}}{r}+\frac{e^{-(r+\rho)}}{r}+\frac{r-\rho}{\rho^2}\Big)\\
&\leq C\frac{e^{-(r-\rho)}}{\sqrt{1-e^{-(r-\rho )}}}\Big(e^{-2\rho}+\frac{1+r-\rho}{\rho^2}\Big),\quad 1\leq \rho \leq r,\;s\in (0,1).
\end{align*}
We have that
\begin{align*}
    \int_R^\infty\int_\rho ^\infty \frac{e^{-(r-\rho)}}{\sqrt{1-e^{-(r-\rho )}}}\Big(e^{-2\rho}+\frac{1+r-\rho}{\rho^2}\Big) drd\rho&\\
    &\hspace{-5cm}=\int_R^\infty e^{-2\rho}d\rho\int_0^\infty \frac{e^{-t}}{\sqrt{1-e^{-t}}}dt+\int_R^\infty\frac{d\rho}{\rho^2}\int_0^\infty \frac{(1+t)e^{-t}}{\sqrt{1-e^{-t}}}dt<\infty.
    &\hspace{-7cm}
\end{align*}
Hence, using again the dominated convergence theorem we deduce that, there exists $\widehat{\alpha} _{n,p}^{(1)}\in \mathbb R$ such that
$$
\lim_{s\rightarrow 0^+}\mathfrak D_1(s)=\widehat\alpha_{n,p}^{(1)}\Phi_p(f(x)).
$$
Meanwhile, from \eqref{J3k5} and \eqref{Aneven}, we get
\begin{align*}
   \lim_{s\rightarrow 0^+}\mathfrak D_2(s)&=\Phi_p(f(x))\frac{c_n\sqrt{\pi}\Gamma (n-1)}{2^{n-3/2}p\sqrt{n-1}\Gamma (n-\frac{1}{2})}\sum_{k=1}^m|c_k^{m,k}|\big(\tfrac{2m-1}{2}\big)^k\lim_{s\rightarrow 0^+}\frac{a_{n,s,p}R^{-\frac{sp}{2}}-c_{n,p}}{s}\\
   &=(\widehat{\alpha}_{n,p}^{(2)}+\widehat{\beta}_{n,p}\log R)\Phi_p(f(x)),
\end{align*}
for certain $\widehat{\alpha}_{n,p}^{(2)},\widehat{\beta}_{n,p}\in \mathbb R$. Thus, we have established property {\it (a)} when $n$ is even. \qed
\\

{\it (b)} Let $x\in \mathbb H^n$ and $R\geq 1$ such that $\supp f\subseteq B(x,R)$. On this occasion we decompose $(-\Delta_{\mathbb H^n})_p^sf(x)$, $s\in (0,\tfrac{\alpha}{p'})$, as follows:
\begin{align*}
(-\Delta_{\mathbb H^n})_p^sf(x)&=\int_{B(x,1)}\Phi_p(f(x)-f(y))\mathcal K_{n,s,p}(d_{\mathbb H^n}(x,y))dy\\
&\quad +\int_{\mathbb H^n\setminus B(x,1)}\big(\Phi_p(f(x)-f(y))-\Phi_p(f(x))\big)\mathcal K_{n,s,p}(d_{\mathbb H^n}(x,y))dy\\
&\quad +\Phi_p(f(x))\int_{\mathbb H^n\setminus B(x,1)}\mathcal K_{n,s,p}(d_{\mathbb H^n}(x,y))dy\\
&=:\sum_{i=1}^3H_i(s),\quad s\in \big(0,\tfrac{\alpha }{p'}\big).
\end{align*}
Thus,
\begin{equation}\label{coclogb}
\frac{(-\Delta_{\mathbb H^n})_p^sf(x)-A_{n,p}\Phi _p(f(x))}{s}=\frac{H_1(s)}{s}+\frac{H_2(s)}{s}+\frac{H_3(s)-A_{n,p}\Phi _p(f(x))}{s},\quad s\in \big(0,\tfrac{\alpha}{p'}\big).
\end{equation}
By proceeding as in the part {\it (a)} with $R=1$ (see \eqref{limI1s}), we get
\begin{equation}\label{H1s}
\lim_{s\rightarrow 0^+}\frac{H_1(s)}{s}=\sigma_{n,p}\int_{B(x,1)}\Phi_p(f(x)-f(y))P_n(d_{\mathbb H^n}(x,y),0)dy.
\end{equation}
According to \cite[Lemmas 2 and 3]{Lin}
\begin{equation}\label{FLin}
\big|\Phi_p(f(x)-f(y))-\Phi_p(f(x))\big|
\leq \left\{\begin{array}{ll}
            (3^{p-1}+2^{p-1})|f(y)|^{p-1},&\mbox{ when }1<p<2,\\[0.2cm]
            (p-1)|f(y)|(|f(x)|+|f(y)|)^{p-2},&\mbox{ when }p\geq 2.
            \end{array}
\right.
\end{equation}
Then, taking into account that $f$ is bounded in $\mathbb H^n$, and \eqref{Kinfty}, we can deduce that
\begin{align*}
\frac{1}{s}\big|[\Phi_p(f(x)-f(y))-\Phi_p(f(x))]\mathcal K_{n,s,p}(d_{\mathbb H^n}(x,y))&\big|\\
&\hspace{-6cm}\leq C(d_{\mathbb H^n}(x,y))^{-1-\frac{sp}{2}}e^{-(n-1)d_{\mathbb H^n}(x,y)}\leq C\frac{e^{-(n-1)d_{\mathbb H^n}(x,y)}}{d_{\mathbb H^n}(x,y)},\quad s\in (0,1),\,y\in \mathbb H^n\setminus B(x,1).
\end{align*}
Since
$$
\int_{\mathbb B(x,R)\setminus B(x,1)}\frac{e^{-(n-1)d_{\mathbb H^n}(x,y)}}{d_{\mathbb H^n}(x,y)}dy\leq C\int_1^R \frac{e^{-(n-1)\rho }(\sinh \rho )^{n-1}}{\rho }d\rho\leq C\int_1^R \frac{d\rho }{\rho }<\infty, 
$$
we can rely on dominated convergence theorem and use \eqref{limKP} to obtain that
\begin{equation}\label{H2s}
\lim_{s\rightarrow 0^+}\frac{H_2(s)}{s}=\sigma_{n,p}\int_{\mathbb H^n\setminus B(x,1)}\big(\Phi_p(f(x)-f(y))-\Phi _p(f(x))\big)P_n(d_{\mathbb H^n}(x,y),0)dy.
\end{equation}
Finally, following the analysis of $D(s)$ (for $R=1$) made in the proof of {\it (a)} (sections \ref{Dodd} and \ref{Deven}) we find a constant $\alpha _{n,p}\in \mathbb R$ such that
\begin{equation}\label{H3s}
\lim_{s\rightarrow 0^+}\frac{H_3(s)-A_{n,p}\Phi_p(f(x))}{s}=\alpha_{n,p}\Phi _p(f(x)).
\end{equation}
Combining \eqref{coclogb}, \eqref{H1s}, \eqref{H2s}, and \eqref{H3s} we finish the proof.
\qed

\section{Proof of Theorem \ref{Th1.3}}\label{S5}
We decompose $E_n(\Phi_p(f(x)-f(\cdot))-\Phi _p(f(x)))(x,t)$, $x\in \mathbb H^n$ and $t\in (0,1)$, as follows
\begin{align}\label{Epn}
 E_n\big(\Phi_p(f(x)-f(\cdot))-\Phi _p(f(x))\big)(x,t)&\nonumber\\
 &\hspace{-6cm}=\left(\int_{\mathbb H^n\setminus B(x,1)}+ \int_{B(x,1)\setminus B(x,t)}+\int_{B(x,t)}\right)\big(\Phi _p(f(x)-f(y))-\Phi _p(f(x))\big)P_n(d_{\mathbb H^n}(x,y),t)dy\nonumber\\
 &\hspace{-5cm}=J_1(x,t)+J_2(x,t)+J_3(x,t),\quad x\in \mathbb H^n\mbox{ and }t\in (0,1).
\end{align}
Let $x\in \mathbb H^n$, and $R\geq 1$ such that ${\rm supp }f\subseteq B(x,R)$. Using \eqref{formderiv} and Lemma \ref{FadiBruno} we get
$$
P_n(\rho ,t)=\hspace{-0.2cm}\sum_{(k, s_1,\ldots,s_k,\ell )\in \Omega_m}\hspace{-0.8cm}c_{k,s_1,\ldots, s_k,\ell}^{(m)}\frac{T_{m,k}(\rho )}{(\sinh \rho )^{2m-k}}\frac{e^{-m\sqrt{\rho ^2+t^2}}}{(\rho ^2+t^2)^{k+\frac{1}{2}-\frac{1}{2}(\sum_{j=1}^ks_j-\ell )}},\quad \rho >0,\,t\geq 0,
$$
if $n=2m+1$, $m\in \mathbb N$, and 
\begin{align*}
P_n(\rho ,t)&=\hspace{-0.2cm}\sum_{(k, s_1,\ldots,s_k,\ell )\in \Omega_m}\hspace{-0.8cm}\widetilde{c}_{k,s_1,\ldots, s_k,\ell}^{(m)}\\
&\quad \times \int_\rho ^\infty \frac{\sinh r}{\sqrt{\cosh r-\cosh \rho }}\frac{T_{m,k}(r)}{(\sinh r)^{2m-k}}\frac{e^{-\frac{2m-1}{2}\sqrt{r^2+t^2}}}{(r^2+t^2)^{k+\frac{1}{2}-\frac{1}{2}(\sum_{j=1}^ks_j-\ell )}}dr,\quad \rho >0,\,t\geq 0,
\end{align*}
 when $n=2m$, $m\in \mathbb N$. Here $c_{k,s_1,\ldots,s_k,\ell}^{(m)},\,\widetilde{c}_{k,s_1,\ldots, s_k,\ell}^{(m)}\in \mathbb R$, for each $(k,s_1,\ldots, s_k,\ell )\in \Omega_m$, being 
$$
 \Omega_m=\Big\{(k,s_1,\ldots, s_k,\ell)\in (\mathbb N\cup\{0\})^{k+2}:k=1,\ldots, m,\,\sum_{j=1}^kjs_j=k,\,\ell =0,\ldots, \sum_{j=1}^ks_j\Big\}.
$$
To simplify the notation, for each $(k,s_1,\ldots,s_k,\ell)\in \Omega _m$, we define
$$
\mathbb T_{k,s_1,\ldots, s_k,\ell}(u,t):=\frac{T_{m,k}(u)}{(\sinh u)^{2m-k}} \frac{e^{-m\sqrt{u^2+t^2}}}{(u^2+t^2)^{k+\frac{1}{2}-\frac{1}{2}(\sum_{j=1}^ks_j-\ell )}},\quad u>0,t\geq 0.
$$
Thus, for each $\rho >0$ and $t\geq 0$,
\begin{equation}\label{Pnodd}
P_n(\rho ,t)=\hspace{-0.2cm}\sum_{(k, s_1,\ldots,s_k,\ell )\in \Omega_m}\hspace{-0.8cm}c_{k,s_1,\ldots, s_k,\ell}^{(m)}\mathbb T_{k,s_1,\ldots, s_k,\ell}(\rho, t),
\end{equation}
if $n=2m+1$, $m\in \mathbb N$, and 
\begin{align}\label{Pneven}
P_n(\rho ,t)&=\hspace{-0.2cm}\sum_{(k, s_1,\ldots,s_k,\ell )\in \Omega_m}\hspace{-0.8cm}\widetilde{c}_{k,s_1,\ldots, s_k,\ell}^{(m)}\int_\rho ^\infty \frac{\sinh r}{\sqrt{\cosh r-\cosh \rho }}e^{\frac{\sqrt{r^2+t^2}}{2}}\mathbb T_{k,s_1,\ldots, s_k,\ell}(r,t)dr,
\end{align}
 when $n=2m$, $m\in \mathbb N$.
 
Also, we will represent by $F_p$ the function defined by
$$
F_p(\rho ,w)=(\Phi _p(f(x)-f(y))-\Phi _p(f(x)))_{|y=\tau _x z,\,z=(\cosh \rho ,w\sinh \rho )}, \quad \rho >0,\,w\in S^{n-1}.
$$ 
Here $\tau _x\in {\rm SO}(1,n)$ is such that $\tau_xe_0=x$.

\noindent Finally, we consider the set of multi-indices $\Lambda_m$, $m\in \mathbb N$, given by 
$$
\Lambda_m=\Big\{(k,s_1,\ldots,s_m,\ell)\in \Omega_m: k=m,\mbox{ and }\ell =\sum_{j=1}^ms_j\Big\},
$$ 
which is needed to deal with $J_2$ and $J_3$.
 
\subsection{Analysis of $J_1$}
Applying Lemma \ref{formderiv} $(iii)$ we obtain, for each $m\in \mathbb N$, 
$$
|P_{2m+1}(\rho ,t)|\leq C\sum_{k=1}^m\frac{\rho ^ke^{-2m\rho }}{(\rho ^2+t^2)^{(k+1)/2}}\leq C\frac{e^{-2m\rho}}{\rho },\quad \rho \geq 1, \,t\geq 0,
$$
and, taking into account also \eqref{sinsin}, it follows that 
\begin{align*}
|P_{2m}(\rho, t)|&\leq C\sum_{k=1}^m\int_\rho ^\infty \frac{\sinh r}{\sqrt{\cosh r-\cosh \rho }}\frac{r^ke^{-(2m-\frac{1}{2})r }}{(r^2+t^2)^{(k+1)/2}}dr\\
&\leq C\int_\rho ^\infty \frac{e^{-(2m-1)r}}{\sqrt{1-e^{-(r-\rho )}}}\frac{dr}{r}\leq C\frac{e^{-(2m-1)\rho }}{\rho}\int_0^\infty \frac{e^{-(2m-1)u}}{\sqrt{1-e^{-u}}}du\\
&\leq C\frac{e^{-(2m-1)\rho }}{\rho},\quad \rho \geq 1, \,t\geq 0.
\end{align*}
We have that
\begin{equation}\label{cotaPn}
|P_n(\rho ,t)|\leq C\frac{e^{-(n-1)\rho}}{\rho },\quad \rho \geq 1, \,t\geq 0. 
\end{equation}
and 
$$
\int_1^R\frac{e^{-(n-1)\rho}(\sinh \rho)^{n-1}}{\rho }d\rho <\infty .
$$ 
Since $f$ is bounded in $\mathbb H^n$, by using \eqref{FLin}, the dominated convergence theorem leads to 
\begin{equation}\label{J1}
\lim_{t\rightarrow 0^+}J_1(x,t)=\int_{\mathbb H^n\setminus B(x,1)}\big(\Phi _p(f(x)-f(y))-\Phi _p(f(x))\big)P_n(d_{\mathbb H^n}(x,y),0)dy.
\end{equation}
\qed

\subsection{Analysis of $J_3$} Our objective is to show that for certain $\mu_n\in \mathbb R$,
\begin{equation}\label{J3}
\lim_{t\rightarrow 0^+}J_3(x,t)=\mu_n \Phi _p(f(x)).
\end{equation}
Consider first $n=2m+1$, $m\in \mathbb N$. Taking into account \eqref{Pnodd} we write
\begin{align*}
J_3(x,t)&=\int_0^t\int_{S^{n-1}}F_p(\rho ,w)P_n(\rho ,t)(\sinh \rho )^{2m}d\sigma (w)d\rho\\
&=\sum_{(k,s_1,\ldots,s_k,\ell)\in \Omega_m}c_{k,s_1,\ldots, s_k,\ell}^{(m)}\int_0^t\int_{S^{n-1}}F_p(\rho , w)\mathbb T_{k,s_1,\ldots,s_k,\ell}(\rho, t)(\sinh \rho )^{2m}d\sigma(w)d\rho\\
&=J_{3,1}^{\rm o}(x,t)+J_{3,2}^{\rm o}(x,t),\quad t\in (0,1),
\end{align*}
where
\begin{align*}
J_{3,1}^{\rm o}(x,t)&=\sum_{(k,s_1,\ldots,s_k,\ell)\in \Omega_m\setminus \Lambda _m}c_{k,s_1,\ldots, s_k,\ell}^{(m)}\\
&\quad \times\int_0^t\int_{S^{n-1}}F_p(\rho , w)\mathbb T_{k,s_1,\ldots,s_k,\ell}(\rho, t)(\sinh \rho )^{2m}d\sigma(w)d\rho,\quad t\in (0,1),
\end{align*}
and $J_{3,2}^{\rm o}(x,t)$, $t\in (0,1)$, denotes the corresponding term obtained by restricting the summation to the multi-indices belonging to $\Lambda _m$. 

Observe that, according to Lemma \ref{formderiv} $(i)$, for $(k,s_1,\ldots, s_k,\ell )\in \Omega_m\setminus \Lambda _m$, we have that
\begin{align}\label{cotabbT}
    |\mathbb T_{k,s_1,\ldots, s_k,\ell}(\rho,t)|&\leq \frac{C}{(\rho ^2+t^2)^{k+\frac{1}{2}-\frac{1}{2}(\sum_{j=1}^ks_j-\ell)}}=C\frac{(\rho ^2+t^2)^{m-k+\frac{1}{2}(\sum_{j=1}^ks_j-\ell)-\frac{1}{2}}}{(\rho ^2+t^2)^m}\nonumber\\
    &\leq \frac{C}{(\rho +t)^{2m}}, \quad \rho, t\in (0,1),
\end{align}
where we have used that $m-k+\frac{1}{2}(\sum_{j=1}^ks_j-\ell)-\frac{1}{2}\geq 0$, when $(k,s_1,\ldots, s_k,\ell )\in \Omega_m\setminus \Lambda _m$.

Taking into account \eqref{FLin} we deduce that
$$
|J_{3,1}^{\rm o}(x,t)|\leq C\int_0^t\frac{\rho ^{2m}}{(\rho +t)^{2m}}d\rho\leq Ct,\quad t\in (0,1),
$$
and thus,
\begin{equation}\label{J31odd}
\lim_{t\rightarrow 0^+}J_{3,1}^{\rm o}(x,t)=0.
\end{equation}
On the other hand, if $(k,s_1,\ldots, s_k,\ell)\in \Lambda _m$ we have that 
\begin{align*}
    \mathbb T_{k,s_1,\ldots, s_k,\ell}(\rho,t)(\sinh \rho)^{2m}&=\frac{\rho ^m(\sinh \rho)^me^{-m\sqrt{\rho ^2+t^2}}}{(\rho ^2+t^2)^{m+\frac{1}{2}}}\\
    &=\Big(\frac{\rho ^m(\sinh \rho)^me^{-m\sqrt{\rho ^2+t^2}}}{(\rho ^2+t^2)^{m+\frac{1}{2}}}-\frac{\rho ^{2m}}{(\rho ^2+t^2)^{m+\frac{1}{2}}}\Big)+\frac{\rho ^{2m}}{(\rho ^2+t^2)^{m+\frac{1}{2}}}\\
    &=:H_1(\rho, t)+H_2(\rho ,t),\quad \rho,t\in (0,1),
\end{align*}
and, therefore
$$
J_{3,2}^{\rm o}(x,t)=A_m\int_0^t\int_{S^{n-1}}F_p(\rho ,w)(H_1(\rho ,t)+H_2(\rho, t))d\sigma (w)d\rho,\quad t\in (0,1),
$$
where $A_m=\sum_{(k,s_1,\ldots,s_k,\ell)\in \Lambda_m}c_{k,s_1,\ldots, s_k,\ell}^{(m)}$.
Let us show that 
\begin{equation}\label{H1}
\lim_{t\rightarrow 0^+}\int_0^t\int_{S^{n-1}}F_p(\rho, w)H_1(\rho , t)d\sigma(w)d\rho=0.
\end{equation}
Indeed, since
\begin{align*}
    |H_1(\rho ,t)|&\leq \frac{\rho ^m\big[(\sinh \rho )^m|e^{-m\sqrt{\rho ^2+t^2}}-e^{-m\rho}|+e^{-m\rho }|(\sinh \rho )^m-\rho ^m|+\rho ^m|e^{-m\rho}-1|\big]}{(\rho ^2+t^2)^{m+\frac{1}{2}}}\\
    &\leq C\frac{\rho ^me^{-m\rho }\big[\rho ^m|\sqrt{\rho ^2+t^2}-\rho|+\rho^{m-1}|\sinh \rho -\rho |\big]+\rho^{2m+1}}{(\rho ^2+t^2)^{m+\frac{1}{2}}}\\
    &\leq C\frac{\rho^{2m}}{(\rho +t)^{2m+1}}\Big(\frac{t^2}{\sqrt{\rho ^2+t^2}+\rho}+\rho \Big)\leq C\frac{\rho^{2m}}{(\rho +t)^{2m}}\leq C,\quad \rho,t\in (0,1),
\end{align*}
again \eqref{FLin} leads to 
$$
\Big|\int_0^t\int_{S^{n-1}}F_p(\rho ,w)H_1(\rho ,t)d\sigma(w)d\rho\Big|\leq C\int_0^td\rho=Ct,\quad t\in (0,1),
$$
and hence \eqref{H1} is established.

On the other hand, by using the differentiation theorem we get
\begin{align*}
\lim_{t\rightarrow 0^+}\int_0^t\int_{S^{n-1}}F_p(\rho ,w)H_2(\rho ,t)d\sigma (w)d\rho&=\lim_{t\rightarrow 0^+}\frac{1}{t}\int_0^t\int_{S^{n-1}}F_p(\rho ,w)d\sigma (w)\frac{(\rho/t)^{2m}}{((\rho /t)^2+1)^{m+\frac{1}{2}}}d\rho\\
&\hspace{-3cm}=c_n\int_0^1\frac{z^{2m}}{(z^2+1)^{m+1/2}}dz\Big(\Phi_p(f(x)-f(\tau_xe_0))-\Phi_p(f(x))\Big)\\
&\hspace{-3cm}=-c_n\int_0^1\frac{z^{2m}}{(z^2+1)^{m+1/2}}dz\Phi _p(f(x))=:\gamma_n\Phi_p(f(x)).
\end{align*}
We obtain that
$$
\lim_{t\rightarrow 0^+}J_{3,2}^{\rm o}(x,t)=\gamma_nA_m\Phi _p(f(x)),
$$
which, jointly \eqref{J31odd}, leads to \eqref{J3} for $\mu_n =\gamma_nA_m$, when $n$ is odd.

Assume now that $n=2m$, with $m\in \mathbb N$. We consider
\begin{align}\label{PnPn1Pn2}
P_n(\rho ,t)&=\Big(\int_\rho ^1+\int_1^\infty \Big)\frac{\sinh r}{\sqrt{\cosh r-\cosh\rho }}\Big(-\frac{1}{\sinh r}\frac{\partial }{\partial r}\Big)^{\frac{n}{2}}\big(\mathbb K_{n}(\sqrt{r^2+t^2})\big)dr\nonumber\\
&=P_{n,1}(\rho, t)+P_{n,2}(\rho ,t),\quad \rho \in (0,1),\,t\geq 0,
\end{align}
and decompose $J_3$ as follows:
\begin{align*}
J_3(x,t)&=\int_0^t\int_{S^{n-1}}F_p(\rho ,w)(\sinh \rho )^{2m-1}\big(P_{n,1}(\rho, t)+P_{n,2}(\rho ,t)\big)d\sigma(w)d\rho\\
&=:J_{3,1}^{\rm e}(x,t)+J_{3,2}^{\rm e}(x,t),\quad t\in (0,1).
\end{align*}
Proceeding as in the estimation \eqref{cotaPn} for $n$ even we get
$$
|J_{3,2}^{\rm e}(x,t)|\leq C\int_0^t\int_{S^{n-1}}|F_p(\rho , w)|\rho ^{2m-2}e^{-(2m-1)\rho}d\sigma (w)d\rho,\quad t\in (0,1).
$$
Then, since $f$ is bounded in $\mathbb H^n$ by virtue of \eqref{FLin} we deduce that
$$
|J_{3,2}^{\rm e}(x,t)|\leq C\int_0^t\rho ^{2m-2}d\rho =Ct^{2m-1},\quad t\in (0,1),
$$
and hence,
\begin{equation}\label{J32even}
\lim_{t\rightarrow 0^+}J_{3,2}^{\rm e}(x,t)=0.
\end{equation}
Using \eqref{Pneven} we decompose $J_{3,1}^{\rm e}$ in the following way 
\begin{align*}
J_{3,1}^{\rm e}(x,t)&=\Big(\sum_{(k,s_1,\ldots,s_k,\ell)\in \Omega_m\setminus \Lambda _m}+\sum_{(k,s_1,\ldots, s_k,\ell)\in \Lambda _m}\Big)\widetilde{c}_{k,s_1,\ldots, s_k,\ell}^{(m)}\int_0^t\int_{S^{n-1}}F_p(\rho , w)d\sigma(w)\\
&\quad \times\int_\rho ^1\frac{\sinh r}{\sqrt{\cosh r -\cosh \rho}}e^{\frac{\sqrt{r^2+t^2}}{2}}\mathbb T_{k,s_1,\ldots,s_k,\ell}(r, t)dr(\sinh \rho )^{2m-1}d\rho\\
&=:J_{3,1,1}^{\rm e}(x,t)+J_{3,1,2}^{\rm e}(x,t),\quad t\in (0,1).
\end{align*}

From \eqref{cotabbT} and the boundedness of $f$ we deduce that
\begin{align}\label{J311even}
    |J_{3,1,1}^{\rm e}(x,t)|&\leq C\int_0^t\int_\rho ^1\frac{\sinh r}{\sqrt{\sinh (\frac{r+\rho}{2})\sinh(\frac{r-\rho}{2})}}\frac{dr}{(r+t)^{2m}}(\sinh \rho )^{2m-1}d\rho\nonumber\\
    &\leq C\int_0^t\rho ^{2m-1}\int_\rho ^1\frac{r}{\sqrt{(r+\rho)(r-\rho)}(r+t)^{2m}}drd\rho\leq C\int_0^t\frac{\rho ^{2m-1}}{\sqrt{\rho}(\rho +t)^{2m-1}}\int_0^{1-\rho}\frac{du}{\sqrt{u}}d\rho\nonumber\\
    &\leq C\int_0^t\frac{d\rho}{\sqrt{\rho}}\int_0^1\frac{du}{\sqrt{u}}\leq C\sqrt{t},\quad t\in (0,1).
\end{align}
Thus, it follows that
$$
\lim_{t\rightarrow 0^+}J_{3,1,1}^{\rm e}(x,t)=0.
$$
Assume next that $(k,s_1,\ldots, s_k,\ell )\in \Lambda _m$.  We can write
\begin{align}\label{K1K2}
(\sinh \rho )^{2m-1}\int_\rho ^1\frac{\sinh r}{\sqrt{\cosh r -\cosh \rho}}e^{\frac{\sqrt{r^2+t^2}}{2}}\mathbb T_{k,s_1,\ldots,s_k,\ell}(r, t)dr\nonumber\\
&\hspace{-8cm}=\left((\sinh \rho )^{2m-1}\int_\rho ^1\frac{\sinh r}{\sqrt{\cosh r -\cosh \rho}}\frac{r^me^{-(m-\frac{1}{2})\sqrt{r^2+t^2}}}{(\sinh r)^m(r^2+t^2)^{m+\frac{1}{2}}}dr\right.\nonumber\\
&\hspace{-8cm}\quad \left. -\sqrt{2}\rho ^{2m-1}\int_\rho ^\infty \frac{r}{\sqrt{r^2-\rho ^2}(r^2+t^2)^{m+\frac{1}{2}}}dr\right)+\sqrt{2}\rho ^{2m-1}\int_\rho ^\infty \frac{r}{\sqrt{r^2-\rho ^2}(r^2+t^2)^{m+\frac{1}{2}}}dr\nonumber\\
&\hspace{-8cm}=:R_1(\rho , t)+ R_2(\rho , t),\quad \rho ,t\in (0,1),
\end{align}
and thus,
$$
J_{3,1,2}^{\rm e}(x,t)=\widetilde{A}_m\int_0^t\int_{S^{n-1}}F_p(\rho ,w)d\sigma (w)(R_1(\rho , t)+R_2(\rho ,t ))d\rho,\quad t\in (0,1).
$$
Here $\widetilde{A}_m=\sum_{(k,s_1,\ldots,s_k,\ell)\in \Lambda_m}\widetilde{c}_{k,s_1,\ldots, s_k,\ell}^{(m)}$. We have that
\begin{align*}
    |R_1(\rho ,t)|&\leq (\sinh \rho )^{2m-1}\int_\rho ^1\frac{\sinh r}{\sqrt{\cosh r -\cosh \rho}}\frac{r^m}{(\sinh r)^m(r^2+t^2)^{m+\frac{1}{2}}}\big|e^{-(m-\frac{1}{2})\sqrt{r^2+t^2}}-1\big|dr\\
    &\quad +|(\sinh \rho )^{2m-1}-\rho ^{2m-1}|\int_\rho ^1\frac{\sinh r}{\sqrt{\cosh r -\cosh \rho}}\frac{r^m}{(\sinh r)^m(r^2+t^2)^{m+\frac{1}{2}}}dr\\
    &\quad +\rho ^{2m-1}\int_\rho ^1\frac{|\sinh r-r|}{\sqrt{\cosh r -\cosh \rho}}\frac{r^m}{(\sinh r)^m(r^2+t^2)^{m+\frac{1}{2}}}dr\\
    &\quad +\rho ^{2m-1}\int_\rho ^1\Big|\frac{1}{\sqrt{\cosh r -\cosh \rho}}-\frac{\sqrt{2}}{\sqrt{r^2-\rho ^2}}\Big|\frac{r^{m+1}}{(\sinh r)^m(r^2+t^2)^{m+\frac{1}{2}}}dr\\
    &\quad +\sqrt{2}\rho ^{2m-1}\int_\rho ^1\Big|\frac{r^m}{(\sinh r)^m}-1{}\Big|\frac{r}{\sqrt{r^2-\rho ^2}(r^2+t^2)^{m+\frac{1}{2}}}dr\\
    &\quad +\sqrt{2}\rho ^{2m-1}\int_1^\infty \frac{r}{\sqrt{r^2-\rho ^2}(r^2+t^2)^{m+\frac{1}{2}}}dr,\quad \rho,t\in (0,1).
\end{align*}
Taking into account that $|1-e^{-z}|\leq Cz$, $|(\sinh z)^\ell -z^\ell |\leq Cz^{\ell +2}$, $z\in (0,1)$ and $\ell \in \mathbb N$, and that
\begin{align*}
\Big|\frac{1}{\sqrt{\sinh z \sinh w}}-\frac{1}{\sqrt{zw}}\Big|&=\frac{|zw-\sinh z\sinh w|}{\sqrt{zw\sinh z\sinh w}(\sqrt{zw}+\sqrt{\sinh z\sinh w})}\\
&\hspace{-2cm}\leq C\frac{z|w-\sinh w|+w|z-\sinh z|}{(zw)^{3/2}}\leq C\Big(\frac{w^{3/2}}{\sqrt{z}}+\frac{z^{3/2}}{\sqrt{w}}\Big)\leq C\frac{w^2+z^2}{\sqrt{zw}},\quad z,w\in (0,1),
\end{align*}
we get
\begin{align}\label{cotaK1}
    |R_1(\rho ,t)|&\leq C\rho ^{2m-1}\left(\int_\rho ^1\frac{r}{\sqrt{r^2-\rho ^2}(r^2+t^2)^m}dr+\rho ^2\int_\rho ^1\frac{r}{\sqrt{r^2-\rho ^2}(r^2+t^2)^{m+\frac{1}{2}}}dr\right.\nonumber\\
    &\quad +\int_\rho ^1\frac{r^3}{\sqrt{r^2-\rho ^2}(r^2+t^2)^{m+\frac{1}{2}}}dr+\int_\rho ^1\frac{(r-\rho)^2+(r+\rho)^2}{\sqrt{r^2-\rho^2}}\frac{r}{(r^2+t^2)^{m+\frac{1}{2}}}dr\nonumber\\
    &\quad \left.+\int_1^\infty \frac{r}{\sqrt{r^2-\rho ^2}(r^2+t^2)^{m+\frac{1}{2}}}dr\right)\nonumber\\
    &\leq C\frac{\rho ^{2m-1}}{\sqrt{\rho }}\left(\Big(\frac{1}{(\rho +t)^{2m-1}}+\frac{1}{(\rho +t)^{2m-2}}\Big)\int_\rho ^1\frac{dr}{\sqrt{r-\rho}}+\int_1^\infty \frac{dr}{r^{2m}}\right)\nonumber\\
    &\leq C\frac{\rho ^{2m-3/2}}{(\rho +t)^{2m-1}}\leq \frac{C}{\sqrt{\rho }},\quad \rho\in \big(0,\tfrac{1}{2}\big),\,t\in (0,1).
\end{align}
Then, 
$$
\Big|\int_0^t\int_{S^{n-1}}F_p(\rho ,w)d\sigma (w)R_1(\rho ,t)d\rho\Big|\leq C\int_0^t\frac{d\rho}{\sqrt{\rho }}=C\sqrt{t},\quad t\in \big(0,\tfrac{1}{2}\big),
$$
and consequently,
\begin{align*}
\lim_{t\rightarrow 0^+}J_{3,1}^{\rm e}(x,t)&=\lim_{t\rightarrow 0^+}J_{3,1,2}^{\rm e}(x,t)=\widetilde{A}_m\lim_{t\rightarrow 0^+}\int_0^t\int_{S^{n-1}}F_p(\rho ,w)d\sigma (w)R_2(\rho ,t )d\rho\\
&\hspace{-1cm}=\sqrt{2}\widetilde{A}_m\lim_{t\rightarrow 0^+}\int_0^t\int_{S^{n-1}}F_p(\rho ,w)d\sigma (w)\int_\rho ^\infty \frac{r}{\sqrt{r^2-\rho ^2}(r^2+t^2)^{m+\frac{1}{2}}}dr\rho ^{2m-1}d\rho\\
&\hspace{-1cm}=\sqrt{2}\widetilde{A}_m\lim_{t\rightarrow 0^+}\frac{1}{t}\int_0^t\int_{S^{n-1}}F_p(\rho, w)d\sigma (w)\int_{\rho/t}^\infty \frac{u}{\sqrt{u^2-(\rho/t)^2}(u^2+1)^{m+\frac{1}{2}}}du\Big(\frac{\rho}{t}\Big)^{2m-1}d\rho\\
&\hspace{-1cm}=\sqrt{2}\widetilde{A}_m\lim_{t\rightarrow 0^+}\int_0^1\int_{S^{n-1}}F_p(tz,w)d\sigma (w)\int_z^\infty \frac{u}{\sqrt{u^2-z^2}(u^2+1)^{m+\frac{1}{2}}}duz^{2m-1}dz.
\end{align*}
Since
$$
\int_0^1\int_z^\infty \frac{u}{\sqrt{u^2-z^2}(u^2+1)^{m+\frac{1}{2}}}duz^{2m-1}dz=:\eta_n<\infty,
$$
the dominated convergence theorem leads to
$$
\lim_{t\rightarrow 0^+}J_{3,1}^{\rm e}(x,t)=-\sqrt{2}\widetilde{A}_mc_n\eta_n\Phi_p(f(x)),
$$
which, jointly \eqref{J32even}, allow us to conclude that \eqref{J3} holds for $\mu_n=-\sqrt{2}\widetilde{A}_mc_n\eta_n$, when $n$ is even.
\qed

\subsection{Analysis of $J_2$} We are going to show that for certain $\theta_n$ and $\gamma_n \in \mathbb R$, 
\begin{equation}\label{J2}
\lim_{t\rightarrow 0^+}\big(J_2(x,t)+\gamma_n \Phi _p(f(x))\log t\big)=\int_{B(x,1)}\Phi _p(f(x)-f(y))P_n(d_{\mathbb H^n}(x,y),0)dy+\theta_n\Phi _p(f(x)).
\end{equation}
First observe that, since $f\in {\rm Lip}_{\rm loc}^\alpha (\mathbb H^n)$ we can write
\begin{align}\label{PhipLip}
    |G_p(\rho, w)|&:=|\Phi _p(f(x)-f(y))|_{|y=\tau_x z,\,z=(\cosh \rho,w\sinh \rho)}=|f(x)-f(y)|^{p-1}_{|y=\tau_x z,\,z=(\cosh \rho,w\sinh \rho)}\nonumber\\
    &\leq Cd_{\mathbb H^n}(\tau _xe_0,\tau _x z)_{|z=(\cosh \rho,w\sinh \rho)}^{\alpha (p-1)}\leq C\rho ^{\alpha (p-1)},\quad \rho \in (0,1),\,w\in S^{n-1}.
\end{align}

Let us assume now that $n=2m+1$, $m\in \mathbb N$, and decompose $J_2$ as follows: 
\begin{align}\label{J2decomp}
J_2(x,t)&=\int_{B(x,1)\setminus B(x,t)}\Phi _p(f(x)-f(y))P_n(d_{\mathbb H^n}(x,y),t)dy\nonumber\\
&\quad -\Phi_p(f(x))\int_{B(x,1)\setminus B(x,t)}P_n(d_{\mathbb H^n}(x,y),t)dy\nonumber\\
&=:J_{2,1}^{\rm o}(x,t)-\Phi _p(f(x))J_{2,2}^{\rm o}(x,t),\quad t\in (0,1).
\end{align}By \eqref{Pnodd} and Lemma \ref{formderiv} $(i)$ we have that
$$
|P_n(\rho ,t)|\leq \frac{C}{(\rho^2+t^2)^{m+\frac{1}{2}}},\quad \rho, t\in (0,1).
$$
Hence, 
$$
|G_p(\rho,w)P_n(\rho , t)(\sinh\rho )^{2m}|\leq C\rho ^{\alpha (p-1)-1},\quad \rho, t\in (0,1),
$$
and the dominated convergence theorem leads to
\begin{equation}\label{J21}
\lim_{t\rightarrow 0^+}J_{2,1}^{\rm o}(x,t)=\int_{B(x,1)}\Phi _p(f(x)-f(y))P_n(d_{\mathbb H^n}(x,y),0)dy.    
\end{equation}
On the other hand, from \eqref{Pnodd} and, as in the analysis of $J_3$, we consider the decomposition
\begin{align}\label{J22decom}
J_{2,2}^{\rm o}(x,t)&=c_n\int_t^1P_n(\rho ,t)(\sinh \rho )^{2m}d\rho\nonumber\\
&=c_n\Big(\sum_{(k,s_1,\ldots,s_k,\ell)\in \Omega_m\setminus \Lambda _m}+\sum_{(k,s_1,\ldots,s_k,\ell)\in \Lambda _m}\Big)c_{k,s_1,\ldots, s_k,\ell}^{(m)}\int_t^1\mathbb T_{k,s_1,\ldots,s_k,\ell}(\rho, t)(\sinh \rho )^{2m}d\rho\nonumber\\
&=J_{2,2,1}^{\rm o}(x,t)+J_{2,2,2}^{\rm o}(x,t),\quad t\in (0,1).
\end{align}
According to \eqref{cotabbT}, for every $(k,s_1,\ldots, s_k,\ell)\in \Omega_m\setminus \Lambda _m$,
$$
|\mathbb T_{k,s_1,\ldots, s_k,\ell}(\rho, t)(\sinh \rho )^{2m}|\leq C\frac{\rho ^{2m}}{(\rho +t)^{2m}}\leq C,\quad \rho, t\in (0,1).
$$
Then, dominated convergence theorem allows us to obtain that, for certain $\delta_n^{(1)}\in \mathbb R$, 
\begin{equation}\label{J221odd}
\lim_{t\rightarrow 0^+}J_{2,2,1}^{\rm o}(x,t)=\delta_n^{(1)}.
\end{equation}
To deal with $J_{2,2,2}^{\rm o}$ we write
\begin{align*}
    J_{2,2,2}^{\rm o}(x,t)&=c_nA_m\int_t^1\frac{\rho ^m(\sinh \rho )^me^{-m\sqrt{\rho ^2+t^2}}}{(\rho ^2+t^2)^{m+\frac{1}{2}}}d\rho\\
    &=c_nA_m\int_t^1\Big(\frac{\rho ^m(\sinh \rho )^me^{-m\sqrt{\rho ^2+t^2}}}{(\rho ^2+t^2)^{m+\frac{1}{2}}}-\frac{\rho ^{2m}}{(\rho ^2+t^2)^{m+\frac{1}{2}}}\Big)d\rho+c_nA_m\psi (t),\quad t\in (0,1),
\end{align*}
where $A_m=\sum_{(k,s_1,\ldots,s_k,\ell)\in \Lambda_m}c_{k,s_1,\ldots, s_k,\ell}^{(m)}$, and $\psi$ is the function defined in Lemma \ref{psivarphi}.
We have that
\begin{align*}
\left|\frac{\rho ^m(\sinh \rho )^me^{-m\sqrt{\rho ^2+t^2}}}{(\rho ^2+t^2)^{m+\frac{1}{2}}}-\frac{\rho ^{2m}}{(\rho ^2+t^2)^{m+\frac{1}{2}}}\right|&\\
&\hspace{-5cm}\leq \frac{\rho ^m}{(\rho ^2+t^2)^{m+\frac{1}{2}}}\Big((\sinh \rho )^m|e^{-m\sqrt{\rho ^2+t^2}}-1|+|(\sinh \rho )^m-\rho ^m|\Big)\\
&\hspace{-5cm}\leq \frac{C}{\rho ^{m+1}}(\rho ^m\sqrt{\rho ^2+t^2}+\rho ^{m+2})\leq C,\quad 0<t<\rho<1.
\end{align*}
Again the dominated convergence theorem assures that there exists $\delta_n^{(2)}\in \mathbb R$ such that
\begin{equation}\label{J222odd}
\lim_{t\rightarrow 0^+}(J_{2,2,2}^{\rm o}(x,t)-c_nA_m\psi (t))=\delta_n^{(2)}.
\end{equation}
From \eqref{J2decomp} and \eqref{J22decom} we can write
\begin{align*}
    J_2(x,t)-c_nA_m\Phi _p(f(x))\log t&=J_{2,1}^{\rm o}(x,t)-\Phi _p(f(x)\Big(J_{2,2,1}^{\rm o}(x,t)+J_{2,2,2}^{\rm o}(x,t)-c_nA_m\psi (t)\\
    &\quad +c_nA_m(\psi (t)+\log t)\Big),\quad t\in (0,1).
\end{align*}
Combining \eqref{psi}, \eqref{J21}, \eqref{J221odd}, and \eqref{J222odd}, we conclude that, when $n$ is odd, \eqref{J2} holds for $\gamma_n =-c_nA_m$, and $\theta_n =-(\delta_n^{(1)}+\delta_n^{(2)}+c_nA_m\alpha _m)$ (here $\alpha _m$ is the constant appearing in Lemma \ref{psivarphi}).

Suppose now that $n=2m$, $m\in \mathbb N$. We write
\begin{align}\label{J2even}
J_2(x,t)&=\int_{B(x,1)\setminus B(x,t)}\Phi_p(f(x)-f(y))(P_{n,1}(d_{\mathbb H^n}(x,y),t)
+P_{n,2}(d_{\mathbb H^n}(x,y),t)\big)dy\nonumber\\
&\quad -\Phi _p(f(x))\int_{B(x,1)\setminus B(x,t)}(P_{n,1}(d_{\mathbb H^n}(x,y),t)
+P_{n,2}(d_{\mathbb H^n}(x,y),t)\big)dy\nonumber\\
&=:J_{2,1}^{\rm e}(x,t)+J_{2,2}^{\rm e}(x,t)-\Phi _p(f(x))\big(J_{2,3}^{\rm e}(x,t)+J_{2,4}^{\rm e}(x,t)\big),\quad t\in (0,1),
\end{align}
where $P_{n,1}$ and $P_{n,2}$ are given by \eqref{PnPn1Pn2}. By proceeding as in \eqref{cotaPn} in the even case, we get 
$$
|P_{n,2}(\rho , t)|\leq C\int_1^\infty \frac{e^{-(2m-1)r}}{\sqrt{1-e^{-(r-\rho )}}}\frac{dr}{r}\leq C\frac{e^{-(2m-1)\rho}}{\rho}\int_0^\infty\frac{e^{-(2m-1)u}}{\sqrt{1-e^{-u}}}du\leq \frac{C}{\rho},\quad \rho \in (0,1).
$$
Taking into account \eqref{PhipLip} and that 
$$
\int_0^1\rho ^{\alpha (p-1)-1}(\sinh\rho )^{2m-1}d\rho +\int_0^1(\sinh \rho )^{2m-1}\frac{d\rho}{\rho}<\infty,
$$
applying the dominated convergence theorem we deduce that
\begin{align}\label{J22J24even}
\lim_{t\rightarrow 0^+}\big(J_{2,2}^{\rm e}(x,t)-\Phi _p(f(x))J_{2,4}^{\rm e}(x,t)\big)&\nonumber\\
&\hspace{-4cm}=\int_{B(x,1)}\Phi_p(f(x)-f(y))P_{n,2}(d_{\mathbb H^n}(x,y),0)dy-\eta_n^{(1)}\Phi _p(f(x)),
\end{align}
where
$$
\eta_n^{(1)}=\int_{B(x,1)}P_{n,2}(d_{\mathbb H^n}(x,y),0)dy<\infty .
$$
Note that $\eta_n^{(1)}$ does not depend on $x$. Now, considering \eqref{Pneven} we write
\begin{align*}
J_{2,1}^{\rm e}(x,t)&=\sum_{(k,s_1,\ldots,s_k,\ell)\in \Omega_m}\widetilde{c}_{k,s_1,\ldots, s_k,\ell}^{(m)}\int_t^1\int_{S^{n-1}}G_p(\rho , w)d\sigma(w)\\
&\quad \times\int_\rho ^1\frac{\sinh r}{\sqrt{\cosh r -\cosh \rho}}e^{\frac{\sqrt{r^2+t^2}}{2}}\mathbb T_{k,s_1,\ldots,s_k,\ell}(r, t)dr(\sinh \rho )^{2m-1}d\rho,\quad t\in (0,1).
\end{align*}
By \eqref{PhipLip} and proceeding as in \eqref{J311even}, for every $(k,s_1,\ldots,s_k,\ell)\in \Omega_m\setminus\Lambda _m$ and  $t, \rho \in (0,1)$,
\begin{equation}\label{rho12}
\Big|\int_{S^{n-1}}G_p(\rho , w)d\sigma(w)\int_\rho ^1\frac{\sinh r}{\sqrt{\cosh r -\cosh \rho}}e^{\frac{\sqrt{r^2+t^2}}{2}}\mathbb T_{k,s_1,\ldots,s_k,\ell}(r, t)dr(\sinh \rho )^{2m-1}\Big|\leq C\rho^{\alpha (p-1)-\frac{1}{2}}.
\end{equation}
When $(k,s_1,\ldots,s_k,\ell)\in \Lambda _m$ we obtain
\begin{align*}
    \Big|\int_{S^{n-1}}G_p(\rho , w)d\sigma(w)\int_\rho ^1\frac{\sinh r}{\sqrt{\cosh r -\cosh \rho}}e^{\frac{\sqrt{r^2+t^2}}{2}}\mathbb T_{k,s_1,\ldots,s_k,\ell}(r, t)dr(\sinh \rho )^{2m-1}\Big|&\\
    &\hspace{-11cm}\leq C\rho^{\alpha (p-1)+2m-1}\int_\rho ^1\frac{r}{\sqrt{r^2-\rho ^2}(r^2+t^2)^{m+\frac{1}{2}}}dr\leq C\frac{\rho ^{\alpha (p-1)+2m-1}}{(\rho +t)^{2m}}\int_\rho ^1\frac{dr}{\sqrt{r+\rho}\sqrt{r-\rho}}\\
    &\hspace{-11cm}\leq C\frac{\rho ^{\alpha (p-1)+2m-1}}{(\rho +t)^{2m}}\int_\rho ^1\frac{dr}{(r+\rho)^{\frac{1}{2}-\varepsilon}(r-\rho)^{\frac{1}{2}+\varepsilon}}\leq C\rho^{\alpha(p-1)+\varepsilon-\frac{3}{2}}\int_0^1\frac{du}{u^{\frac{1}
    {2}+\varepsilon}}\\
    &\hspace{-11cm}\leq C\rho^{\alpha (p-1)+\varepsilon-\frac{3}{2}},\quad t,\,\rho \in (0,1), \,\varepsilon \in (0,\tfrac{1}{2}).
\end{align*}
Since $\alpha>0$ and $p>1$, we can choose $\varepsilon \in (0,\tfrac{1}{2})$ such that $\alpha(p-1)>\tfrac{1}{2}-\varepsilon$. Then, the dominated convergence theorem allows us to obtain
\begin{equation}\label{J21even}
    \lim_{t\rightarrow 0^+}J_{2,1}^{\rm e}(x,t)=\int_{B(x,1)}\Phi _p(f(x)-f(y))P_{n,1}(d_{\mathbb H^n}(x,y),0)dy.
\end{equation}
We use again \eqref{Pneven} to decompose $J_{2,3}^{\rm e}$ as follows. 
\begin{align}\label{J23decom}
J_{2,3}^{\rm e}(x,t)&=c_n\Big(\sum_{(k,s_1,\ldots,s_k,\ell)\in \Omega_m\setminus \Lambda _m}+\sum_{(k,s_1,\ldots, s_k,\ell)\in \Lambda _m}\Big)\widetilde{c}_{k,s_1,\ldots, s_k,\ell}^{(m)}\nonumber\\
&\quad \times\int_t^1\int_\rho ^1\frac{\sinh r}{\sqrt{\cosh r -\cosh \rho}}e^{\frac{\sqrt{r^2+t^2}}{2}}\mathbb T_{k,s_1,\ldots,s_k,\ell}(r, t)dr(\sinh \rho )^{2m-1}d\rho\nonumber\\
&=:J_{2,3,1}^{\rm e}(x,t)+J_{2,3,2}^{\rm e}(x,t),\quad t\in (0,1).
\end{align}
As in \eqref{rho12} it follows that, for $(k,s_1,\ldots, s_k,\ell)\in \Omega_m \setminus \Lambda _m$,
$$
\Big|\int_\rho ^1\frac{\sinh r}{\sqrt{\cosh r -\cosh \rho}}e^{\frac{\sqrt{r^2+t^2}}{2}}\mathbb T_{k,s_1,\ldots,s_k,\ell}(r, t)dr(\sinh \rho )^{2m-1}\Big|\leq \frac{C}{\sqrt{\rho}},\quad t,\,\rho \in (0,1).
$$
Hence, for certain $\eta_n^{(2)}\in \mathbb R$, 
\begin{equation}\label{J231even}
\lim_{t\rightarrow 0^+}J_{2,3,1}^{\rm e}(x,t)=\eta_n^{(2)}.
\end{equation}
On the other hand,
$$
J_{2,3,2}^{\rm e}(x,t)=c_n\widetilde{A}_m\int_t^1\big(R_1(\rho, t)+R_2(\rho ,t)\big)d\rho=c_n\widetilde{A}_m\int_t^1R_1(\rho, t)d\rho +\sqrt{2}c_n\widetilde{A}_m\varphi (t),\quad t\in (0,1),
$$
being $\widetilde{A}_m=\sum_{(k,s_1,\ldots,s_k,\ell)\in \Lambda_m}\widetilde{c}_{k,s_1,\ldots, s_k,\ell}^{(m)}$, $R_1$ and $R_2$ as in \eqref{K1K2}, and $\varphi$ the function defined in Lemma \ref{psivarphi}. Then, taking into account \eqref{varphi}, and \eqref{cotaK1} we get
\begin{equation}\label{J232even}
\lim_{t\rightarrow 0^+}\big(J_{2,3,2}^{\rm e}(x,t)+\sqrt{2}c_n\widetilde{A}_m\sigma_m\log t\big)=\eta_n^{(3)}+\sqrt{2}c_n\widetilde{A}_m\lim_{t\rightarrow 0^+}(\varphi (t)+\sigma_m\log t)=\eta_n^{(3)}+\sqrt{2}c_n\widetilde{A}_m\beta_m,
\end{equation}
for certain $\eta_n^{(3)}\in \mathbb R$, and where $\sigma _m$, $\beta_m$ are the constants appearing in Lemma \ref{psivarphi}.

Putting together \eqref{J2even}, \eqref{J22J24even}, \eqref{J21even}, \eqref{J23decom}, \eqref{J231even}, and \eqref{J232even} we obtain that, in the case that $n$ is even, \eqref{J2} is satisfied with $\gamma_n=-\sqrt{2}c_n\widetilde{A}_m\sigma _m$, and $\theta_n=-(\eta_n^{(1)}+\eta_n^{(2)}+\eta_n^{(3)}+\sqrt{2}c_n\widetilde{A}_m\beta_m)$.

\qed

Finally, taking into account Theorem \ref{Th1.2} $(b)$, the decomposition \eqref{Epn} and properties \eqref{J1}, \eqref{J3}, and \eqref{J2} we can finish the proof.

\section{Proof of Theorem \ref{Th1.4}}\label{S6}
Suppose that $f\in C_c^\alpha (\mathbb R^n)$ for some $\alpha >0$. In \cite[Theorem 1.1]{DJF} it was established that
\begin{align}\label{LogDeltap}
    (\log (-\Delta)_p)f(x)&=\eta_{n,p}\int_{B(x,1)}\frac{\Phi _p(f(x)-f(y))}{|x-y|^n}dy\nonumber\\
    &\quad +\eta_{n,p}\int_{\mathbb R^n\setminus B(x,1)}\frac{\Phi _p(f(x)-f(y))-\Phi _p(f(x))}{|x-y|^n}dy +\delta_{n,p}\Phi _p(f(x)),\quad x\in \mathbb R^n,
\end{align}
where $\eta_{n,p}=p\Gamma (n/2)/(2\pi ^{n/2})=p/c_n$, and $\delta_{n,p}=2\log 2-\gamma+p\psi (n/2)/2$, being $\gamma=-\Gamma'(1)$ the Euler-Mascheroni constant and $\psi =\Gamma '/\Gamma$ the Digamma function.

We decompose ${\mathbb E}_n(\Phi _p(f(x)-f(y))-\Phi _p(f(x)))$, $x\in \mathbb R^n$, as follows:
\begin{align*}
{\mathbb E}_n(\Phi _p(f(x)-f(y))-\Phi _p(f(x)))&\\
&\hspace{-4cm}=\frac{1}{c_n}\left(\int_{\mathbb R^n\setminus B(x,1)}+ \int_{B(x,1)\setminus B(x,t)}+\int_{B(x,t)}\right)\frac{\Phi _p(f(x)-f(y))-\Phi _p(f(x))}{(|x-y|^2+t^2)^{n/2}}dy\\
&\hspace{-4cm}=I_1(x,t)+I_2(x,t)+I_3(x,t),\quad x\in \mathbb R^n, \mbox{ and }t\in (0,1).
\end{align*}
Let $x\in \mathbb R^n$. We take $R>0$ such that ${\rm supp }f\subseteq B(x,R)$. We have that
$$
I_1(x,t)=\frac{1}{c_n}\int_{B(x,R)\setminus B(x,1)}\frac{\Phi _p(f(x)-f(y))-\Phi _p(f(x))}{(|x-y|^2+t^2)^{n/2}}dy,\quad t\in (0,1).
$$
By using dominated convergence theorem we obtain
$$
\lim_{t\rightarrow 0^+}I_1(x,t)=\frac{1}{c_n}\int_{\mathbb R^n\setminus B(x,1)}\frac{\Phi _p(f(x)-f(y))-\Phi _p(f(x))}{|x-y|^n}dy.
$$
On the other hand, we have that
$$
\left|\frac{\Phi _p(f(x)-f(y))}{(|x-y|^2+t^2)^{n/2}}\right|\leq C\frac{|f(x)-f(y)|^{p-1}}{|x-y|^n}\leq C|x-y|^{\alpha (p-1)-n},\quad y\in \mathbb R^n,\,y\not=x,\mbox{ and }t>0,
$$
and the dominated convergence theorem leads to 
$$
\lim_{t\rightarrow 0^+}\int_{\mathbb B(x,1)\setminus B(x,t)}\frac{\Phi _p(f(x)-f(y))}{(|x-y|^2+t^2)^{n/2}}dy=\int_{B(x,1)}\frac{\Phi _p(f(x)-f(y))}{|x-y|^n}dy,
$$
and
$$
\lim_{t\rightarrow 0^+}\int_{\mathbb B(x,t)}\frac{\Phi _p(f(x)-f(y))}{(|x-y|^2+t^2)^{n/2}}dy=0.
$$
Observe also that
$$
\int_{B(x,t)}\frac{dy}{(|x-y|^2+t^2)^{n/2}}=c_n\int_0^t\frac{\rho ^{n-1}}{(\rho ^2+t^2)^{n/2}}d\rho=c_n\int_0^1\frac{u^{n-1}}{(1+u^2)^{n/2}}du,\quad t>0, 
$$
hence
$$
\lim_{t\rightarrow 0^+}I_3(x,t)=-\widetilde{q}_n\Phi_p(f(x)),
$$
where 
$$
\widetilde{q}_n=\int_0^1\frac{u^{n-1}}{(1+u^2)^{n/2}}du.
$$ 
According to \cite[Lemma 4.1]{CHW} we have
$$
\int_{B(x,1)\setminus B(x,t)}\frac{dy}{(|x-y|^2+t^2)^{n/2}}dy=\frac{c_n}{2}\big(-2\log t+q_n+g(t)\big),\quad t>0,
$$
where $\lim_{t\rightarrow 0^+}g(t)=0$, and 
$$
q_n=2\int_1^\infty \Big(\frac{1}{(\rho ^2+1)^{n/2}}-\frac{1}{\rho ^n}\Big)\rho ^{n-1}d\rho.
$$
We deduce that
$$
\lim_{t\rightarrow 0^+}\big(I_2(x,t)-\Phi_p(f(x))\log t\big)=\frac{1}{c_n}\int_{B(x,1)}\frac{\Phi _p(f(x)-f(y))}{|x-y|^n}dy-\frac{q_n}{2}\Phi _p(f(x)).
$$
Thus,  
\begin{align*}
\lim_{t\rightarrow 0^+}\big({\mathbb E}_n(\Phi _p(f(x)-f(\cdot))-\Phi _p(f(x)))-\Phi_p(f(x)\log t)\big)&=-\big(\frac{q_n}{2}+\widetilde{q}_n)\Phi _p(f(x))\\
&\hspace{-8cm}\quad +\frac{1}{c_n}\left(\int_{B(x,1)}\frac{\Phi _p(f(x)-f(y))}{|x-y|^n}dy+\int_{\mathbb R^n\setminus B(x,1)}\frac{\Phi _p(f(x)-f(y))-\Phi _p(f(x))}{|x-y|^n}dy\right).
\end{align*}

Taking into account that $c_n^{-1}=\eta_{n,p}/p$ and using \eqref{LogDeltap} we conclude that the statement in Theorem \ref{Th1.4} holds for $\alpha_{n,p}=\delta_{n,p}+p(\tfrac{q_n}{2}+\widetilde{q}_n)$.


\bibliographystyle{acm}

\end{document}